\documentclass[11pt, a4paper]{article}
 \usepackage{geometry}
\usepackage[english]{babel}

\usepackage{amsmath,amssymb,amsthm,amsthm}
\usepackage{mathrsfs}
\usepackage{enumerate}

\usepackage{todonotes}  

\usepackage[noadjust]{cite} 

\theoremstyle{plain}
\newtheorem{theorem}{Theorem}[section]
\newtheorem*{theorem*}{Theorem} 
\newtheorem{proposition}[theorem]{Proposition}
\newtheorem{lemma}[theorem]{Lemma}

\newtheorem{claim}[theorem]{Claim}

\theoremstyle{definition}
\newtheorem{definition}[theorem]{Definition}
\theoremstyle{definition}

\theoremstyle{definition}
\newtheorem*{definition*}{Definition}

\theoremstyle{definition}
\newtheorem{notation}[theorem]{Notation}
\theoremstyle{definition}
\newtheorem*{notation*}{Notation}

\theoremstyle{definition}
\newtheorem{remark}[theorem]{Remark}
\theoremstyle{definition}
\newtheorem*{remark*}{Remark}
\theoremstyle{definition}
\newtheorem{remarks}[theorem]{Remarks}
\theoremstyle{definition}
\newtheorem*{remarks*}{Remarks}

\theoremstyle{remark}
\theoremstyle{remark}

\newtheorem*{example*}{Example}
\newtheorem*{examples*}{Examples}

\theoremstyle{definition}

\theoremstyle{remark}

\newcommand{\N}{{\mathbb N}}
\newcommand{\R}{{\mathbb R}}

\def\N{\mathbb{N}}

\newcommand{\RR}{{\R}}
\newcommand{\NN}{{\N}}
\newcommand{\EE}{{\E}}

\newcommand{\PP}{{\mathbb P}}
\newcommand{\QQ}{{\mathbb Q}}
\newcommand{\Law}{{\textrm{Law}}}

\newcommand{\KL}{{\mathcal H}} 
\newcommand{\WW}{{\mathbb W}}
\newcommand{\SE}{{\textrm{se}}} 

\usepackage{amsmath}
\usepackage{amsfonts}
\usepackage{amsthm}
\usepackage{amssymb}

 \newcommand{\E}{\mathbb{E}}

\makeindex

\usepackage{tikz}
\usepackage{tikz-cd} 
\usetikzlibrary{arrows}
\tikzcdset{arrow style=tikz, diagrams={>=stealth}} 
\usepackage{graphicx}  

\usetikzlibrary{shapes.geometric,quotes}

\DeclareMathOperator{\Tr}{Tr}

\newcommand{\ii}{\mathrm{i}}
\newcommand{\dd}{\mathrm{d}}

\makeatletter
\newcommand\RedeclareMathOperator{%
  \@ifstar{\def\rmo@s{m}\rmo@redeclare}{\def\rmo@s{o}\rmo@redeclare}%
}
\newcommand\rmo@redeclare[2]{%
  \begingroup \escapechar\m@ne\xdef\@gtempa{{\string#1}}\endgroup
  \expandafter\@ifundefined\@gtempa
     {\@latex@error{\noexpand#1undefined}\@ehc}%
     \relax
  \expandafter\rmo@declmathop\rmo@s{#1}{#2}}
\newcommand\rmo@declmathop[3]{%
  \DeclareRobustCommand{#2}{\qopname\newmcodes@#1{#3}}%
}
\@onlypreamble\RedeclareMathOperator
\makeatother

\RedeclareMathOperator*{\div}{{div}}

\newcommand{\deq}{{\;\stackrel{\mathrm{def}}{=}\;}}

\usepackage{hyperref}

\title{A stochastic approach to time-dependent BEC}
\author{Luigi Borasi\thanks{Bergische Universität Wuppertal, Germany. email:\texttt{borasi@uni-wuppertal.de}}
  \and
Francesco C. De Vecchi\thanks{Universit{\`a} degli Studi di Pavia, Italy. email:\texttt{francescocarlo.devecchi@unipv.it}}
\and Stefania Ugolini\thanks{Universit{\`a} degli Studi di Milano, Italy. email:\texttt{stefania.ugolini@unimi.it}}}
  
\begin{document}

\maketitle

\abstract{We propose a stochastic description of the dynamics of a Bose-Einstein condensate within the context of Nelson stochastic mechanics. 
We start from the $N$ interacting conservative diffusions, associated with the $N$ Bose particles, and take an infinite particle limit. 
We address several aspects of this formulation.
First, we consider the problem of extending to a system with self-interaction the variational formulation of Nelson stochastic mechanics due to Guerra and Morato.
In this regard we discuss two possible extensions, one based on a doubling procedure and another based on a constraint Eulerian type variational principle.
Then we consider the infinite particle limit from the point of view of the $N$-particles Madelung equations.
Since conservative diffusions can be identified with proper infinitesimal characteristics pairs $(\rho_N(t), v_N(t))$, a time marginal probability density and a current velocity field, respectively, we consider a finite Madelung hierarchy for the marginals pairs $(\rho_{N,n}(t), v_{N,n}(t))$, obtained by properly conditioning the processes. The infinite Madelung hierarchy arises from the finite one by performing, for each fixed $n$, a mean-field scaling limit in $N$. 
Finally, we introduce a $n$-particle conditioned diffusions which naturally parallels the quantum mechanical approach and is a new approach within the context of Nelson stochastic mechanics. We then prove the convergence, in the infinite particle limit, 
of the law of such a conditioned process
to the law of a self-interacting diffusion which describes the condensate.}

\section{Introduction}
We focus on providing a dynamical description of a time-dependent Bose-Einstein condensate in terms of a weak-coupling limit of suitable proper infinitesimal characteristics in the class of Nelson conservative diffusion processes. While the problem we are facing has a classical prototype in the derivation of a Vlasov equation from a system of $N$ classical particles, where the convergence of the Liouville to the Vlasov hierarchy has been proved (\cite{Braun}), here  we follow a standard quantum mechanical approach which derives the one-particle nonlinear Schr\"odinger equation starting from  the $N$-body Schr\"odinger equation via the $N$-particle (finite) Schr\"odinger hierarchy and its weak-coupling infinite hierarchy limit. In particular, we consider here the analysis done in \cite{Bardos,S80} (see also \cite{ESY,FGS07,FKP07,KP10,P15,Schlein2009}) in the case of mean field regime, where the convergence  of the (finite) Schr\"odinger hierarchy solution to the solution of the infinite one has been proved under quite general conditions, and, under the hypothesis of bounded potential, the uniqueness, the stability and the factorization of the infinite hierarchy have been discussed. The cited results imply the crucial property that if one starts from a factorized initial marginal density  for the (linear) finite Schr\"odinger hierarchy,
in the infinite scaling limit, the marginal density factorizes  
and it is equal to the product of $n$ copies of the same density,
solving the one-particle nonlinear Schr\"odinger equation.
 
In the present paper, we investigate the probabilistic counterpart of the cited standard quantum framework and we are able to derive a stochastic model for the limit quantum Bose-Einstein condensate dynamics, which is described in terms of the self-interacting conservative diffusion associated with the one-particle nonlinear Schr\"odinger equation.

\medskip

 A stochastic description of a stationary Bose-Einstein Condensate (BEC), within Nelson stochastic mechanics,
 has been provided in 2011 \cite{MorUgo1},
 starting from the ground-state of the $N$ body Hamiltonian for $N$ Bose particles  and by  performing the Gross-Pitaevskii scaling limit of infinitely many particles (\cite{Ugolini}, \cite{albeverio2015doob}).
 All the results regarding the stationary BEC are based on a fruitful approach which considers the relative entripy between the two involved probability measures supported on the process path space.
 An existence theorem for the probability measure associated to the minimizer of the Gross-Pitaevskii functional has been given in \cite{DeVecchiUgolini}, while the property of the asymptotic localization of relative entropy has been discussed in \cite{morato2013localization}.
 The general mean-field convergence problem, using the hard results for the case of Gross-Pitaevskii scaling limit obtained in \cite{PhysRevLett.88.170409}  and, recently,  in \cite{boccato2018complete,nam2016ground}, has been studied in \cite{albeverio2020strong} for the case of purely repulsive interacting potential.
 Moreover, in \cite{albeverio2020strong}, it has been established the convergence in total variation for the mean-field scaling limit and in \cite{albeverio2017entropy} the weak convergence for the Gross-Pitaevskii scaling limit. Furthermore, the entropy chaos, a stronger chaoticity property introduced in \cite{Carlen2010},  has been finally proved in \cite{albeverio2017entropy}.

 \medskip

The physical mathematical setting of time-dependent Bose-Einstein condensation is very different from the stationary scenario because it is no more possible to take advantage of the crucial properties of the ground-state quantum energy, and in particular of the variational feature of the probability distribution as its unique minimizer.

We face the general time-dependent $N$-body BEC dynamics starting from the following specific framework inspired by a first principles approach.
Let $\Psi_N=\Psi_N(x_1,\dots,x_N)$
be the wave function of $N$ interacting Bose particles
which evolve in time according to the following $N$-particles Schr\"odinger equation on $\R^{Nd}$:
\begin{equation}
  \label{schro_intro}  
  \begin{aligned}
    i\hbar \partial_t \Psi_N = & - \frac{\hbar ^2}{2} \sum_{1 \leq j \leq N} \Delta_{x_j} \Psi_N + \frac{1}{N} \sum_{1 \leq j < k \leq N} V(|x_j - x_k|) \Psi_N =:H_N \Psi_N,\\
                               &\Psi_N(t=0) = \Psi_N^I (x_1, x_2, ..., x_N),
  \end{aligned}\tag{S}
\end{equation}
where $H_N$ is the Hamiltonian and the term $\frac{1}{N}$ before the potential $V$ reveals the ``weak coupling scaling''~\cite{QKE} and where
the potential $V$ is assumed to be real-valued and bounded from below.

Since the Schr\"odinger equation for an $N$-particle system is linear,  without any loss of generality, one can consider the density matrix for a pure state:
\begin{equation}\label{eq:1}
  \rho^Q_N(x^N,y^N,t)=\Psi_N(x^N,t)\overline{\Psi_N(y^N,t)},
\end{equation}
where $x^N=(x_1,x_2,\dots,x_N).$

 Under a suitable mean-field infinite particle limit one obtains the self-interacting one-particle nonlinear Schr\"odinger equation on $\R^{d}$:
  \begin{equation}
    \label{nonlin_schro_intro}
    i \hbar \partial_t \psi (x,t)
    = - \frac{\hbar^2}{2} \Delta_x \psi (x,t) + \int V(|x - z|)|\psi (z,t)|^2 \ dz \cdot \psi (x,t)
    ,\tag{NLS}
  \end{equation}

Our goal is to describe the above scaling limit
within Nelson stochastic mechanics, as we already said, inspired by \cite{Bardos}, in which for the first time, as far as we know, the problem is discussed in terms of the quantum density matrices solving finite and infinite Schr\"odinger hierarchies. The cited paper gives rise to a wide research area in mathematical physics,  with the aim of deriving a nonlinear Schr\"odinger equation from the many-body dynamics by using the marginal Schr\"odinger hierarchy (actually known as the BBGKY hierarchy). This research line has gone on to analyze the mean-field scaling limit to then tackle the more difficult case of the Gross-Pitaevskii scaling limit (see \cite{ErdosBardos,ESY,erdos2001derivation,ErdossSchlein}).
The uniqueness of the solution to the infinite hierarchy is the most difficult part of this kind of analysis.

\medskip

Since in the present paper we face for the first time this  problem from the stochastic point of view, we start by the simpler mean-field limit and by assuming, as in \cite{Bardos}, the boundedness of the external potential in order to have the uniqueness of the infinite hierarchy solution.

In the diagram below, we briefly illustrate our main probabilistic description of the dynamics of BEC quantum phenomena.

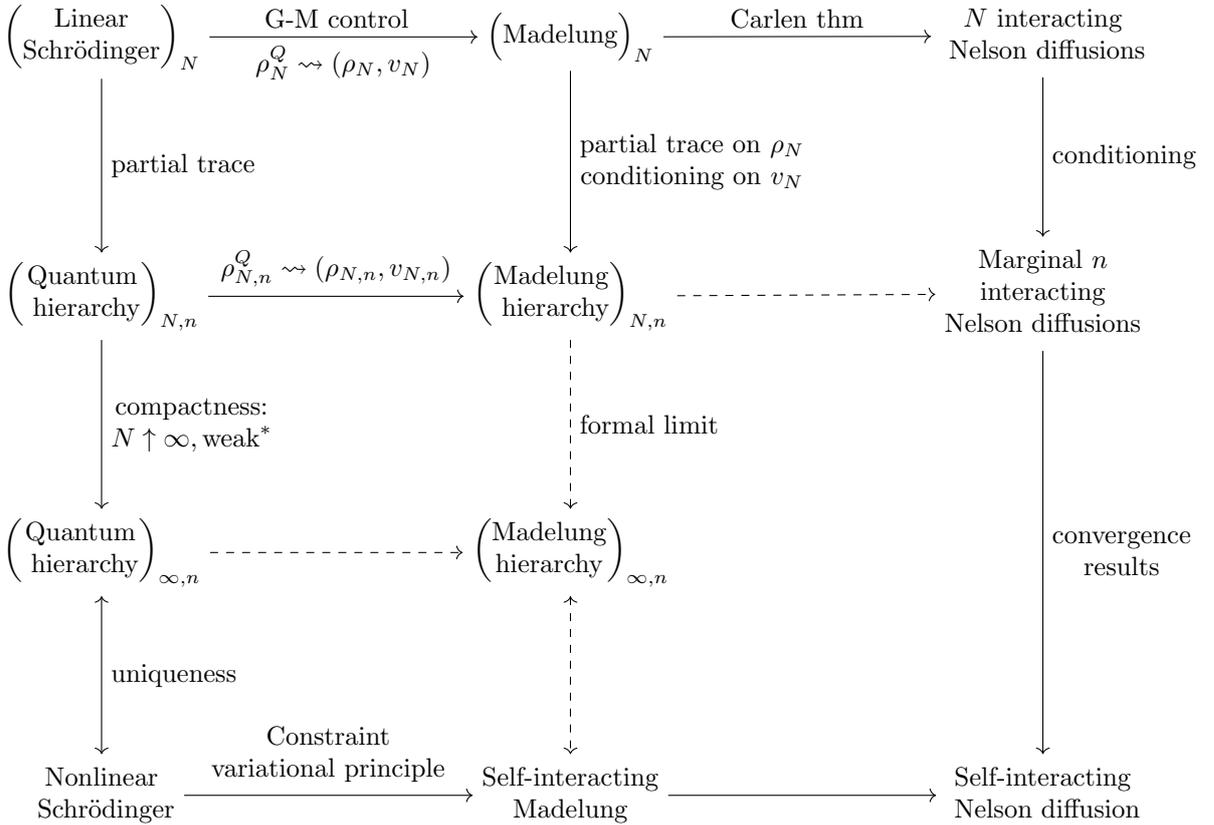
\begin{tikzpicture}
  \matrix (m)
  [
  matrix of math nodes,
  row sep    = 6em,
  column sep = 9.5em 
  ]
  {    
    \begin{pmatrix}\text{Linear}\\ \text{Schr\"odinger}\end{pmatrix}_N   & \Big(\text{Madelung}\Big)_N
    & \begin{matrix} N\text{ interacting} \\ \text{ Nelson diffusions} \end{matrix}\\
    \begin{pmatrix}\text{Quantum}\\ \text{ hierarchy}\end{pmatrix}_{N,n}
    &\begin{pmatrix}\text{Madelung}\\ \text{ hierarchy}\end{pmatrix}_{N,n}
    &  \begin{matrix} \text{Marginal }n \\ \text{interacting} \\ \text{Nelson diffusions} \end{matrix} \\
    \begin{pmatrix}\text{Quantum}\\ \text{ hierarchy}\end{pmatrix}_{\infty,n}
    &\begin{pmatrix}\text{Madelung}\\ \text{hierarchy}\end{pmatrix}_{\infty,n}
    & \begin{matrix} \text{} \\ \text{ } \end{matrix} \\
    \begin{matrix} \text{Nonlinear} \\ \text{ Schr\"odinger} \end{matrix}
    & \begin{matrix} \text{Self-interacting}\\ \text{Madelung}\end{matrix}
    & \begin{matrix} \text{Self-interacting} \\ \text{ Nelson diffusion} \end{matrix} \\
  };
  \path
  (m-1-1) edge [->] node [above] {$\!\!\!$G-M control} node [below] {$\rho^Q_{N}\leadsto(\rho_{N},v_{N})$} (m-1-2)
  (m-1-1) edge [->] node [right] {partial trace} (m-2-1)
  (m-1-2) edge [->] node [right] {$\begin{matrix} \text{partial trace on }\rho_N\\
    \text{conditioning on }v_N \end{matrix}$} (m-2-2)
(m-1-3) edge [->] node [right] {conditioning} (m-2-3)
(m-2-1) edge [->] node [above] {$\rho^Q_{N,n}\leadsto(\rho_{N,n},v_{N,n})$} node [below] {} (m-2-2)
(m-2-1) edge [->] node [right] {$\begin{matrix} \text{compactness:}\\
  N\uparrow\infty, \text{weak}^\ast  \end{matrix}$} (m-3-1)
(m-2-2) edge [->,dashed] node [right] {formal limit} (m-3-2)
(m-2-3) edge [->] node [right] {$\begin{matrix}\text{convergence}\\ \text{results} \end{matrix}$} (m-4-3)
(m-3-1) edge [->,dashed] node [above] {{\color{red}}} node [below] {} (m-3-2)
(m-3-1) edge [<->] node [right] {uniqueness} (m-4-1)
(m-3-2) edge [<->,dashed] node [right] {{\color{red}}} (m-4-2)
(m-4-1) edge [->] node [above] {$\begin{matrix} \text{Constraint} \\ \text{variational principle} \end{matrix}$} node [below] {} (m-4-2)
(m-1-2) edge [->] node [above] {Carlen thm} node [below] {} (m-1-3)
(m-2-2) edge [->,dashed] node [above] {} node [below] {} (m-2-3)
(m-4-2) edge [->] node [above] {} node [below] {} (m-4-3)
;
\end{tikzpicture}\\
On the left-hand
vertical arrow, we have the quantum mechanical approach: by taking partial trace of the N particle quantum density matrix $\rho_N^Q(x,y,t)$ associated with \ref{schro_intro}   one obtains the marginal Schr\"odinger hierarchy solved by the marginal distributions $\rho_{N,n}^Q(x,y,t)$ 
and,  by performing a formal $N$ infinite limit, one gets the infinite Schr\"odinger hierarchy. By compactness of the marginal distributions $\rho_{N,n}^Q(x,y,t)$ for all $t\in \R_+$ in a suitable space and under some reasonable assumptions on the initial conditions and potential, the convergence of the solution of the finite hierarchy to the solution of the infinite hierarchy can be established. For bounded potentials the uniqueness of the factorized solution of the infinite hierarchy is proved where each factor solves the one-particle nonlinear Schr\"odinger equation \ref{nonlin_schro_intro} (\cite{Bardos}). 

The emerging probabilistic structure is conceptualized in the horizontal and in the right-hand vertical arrows.
By Nelson stochastic mechanics to the $N$-particle Schr\"odinger equation \eqref{schro_intro}
there is associated a system of $N$ \textit{interacting conservative diffusions} via Carlen theorem  (Theorem \ref{Carlentheorem}) which are characterized by a proper infinitesimal characteristics pair $ (\rho_N,v_N)$, where $\rho_N(x,t)=|\Psi_N(x,t)|^2$ is a time-dependent probability density and $v_N(x,t)$ a time-dependent vector field,  such that they solve both a weak continuity equation and a finite  $L^2$-type energy condition (Definition \ref{conservativediffusiondefinition} and the subsequent remark), which is also a finite entropy condition (see \cite{Follmer_Wiener_space}).

\medskip

A $N$-dimensional conservative diffusion is the solution to a $N$-dimensional Brownian motion-driven stochastic differential equation (SDE) given by
\begin{equation}\label{NelsonNdiffusions_intro}
dX^N_t=(\frac{\nabla \rho_N}{2\rho_N}(X^N_t)+v_N(X^N_t))dt+dW^N_t,
\end{equation}
with drift expressed in terms of the infinitesimal characteristics pair $ (\rho_N,v_N)$, $W^N_t$ an $N$-dimensional Brownian motion and
where $\rho_N$ satisfies in a weak sense the following continuity equation
\begin{equation}\label{continuity_equation_intro}
 \partial_t \rho_N+\nabla \cdot (\rho_N v_N)=0.  
\end{equation} 
We stress that the conservative diffusion class,  first introduced by Nelson (\cite{NelsonQF}), constitutes a class of processes that have been well studied from the probabilistic point of view, in particular for their time-reversal invariance property (\cite{haussmann1986time,millet1989integration}) and for the properties of the related fundamental Nelson derivative (\cite{Cattiaux2021TimeRO}).
Furthermore, this peculiar class of processes can be obtained by a stochastic variational formulation of stochastic mechanics starting from a Guerra-Morato (GM) Lagrangian (\cite{GuMo}) or from other variational principles (\cite{pavon1995hamilton,pavon1989stochastic,Yasue}).

\medskip

In this paper, we start by providing a brief and informal revisit of the GM stochastic variational principle in terms of the moderne formulation of a McKean-Vlasov control problem (see, e.g., \cite{pham2009continuous}), which allows us to formally derive a system of equations denominated \emph{Madelung equations}. With Madelung equations, we specifically mean a nonlinear PDE system given by a continuity equation and a coupled nonlinear equation for the current velocity.  We stress that the cited stochastic control problem is not an optimal control one, since the cost function introduced starting from the GM Lagrangian is not convex and the variational procedure selects a critical diffusion by imposing only the stationarization of the cost function.  An interesting min-max stochastic variational principle has been proposed within the framework of stochastic mechanics in \cite{pavon2006}.  

By taking jointly the partial trace and introducing a suitable \textit{conditioning procedure} with respect to the first $n$ variables, in the present paper we derive the \textit{(finite) Madelung hierarchy} for the marginal infinitesimal characteristics $ (\rho_{N,n},v_{N,n})$ and, by performing a formal limit $N \to \infty$,
we identify the form of the corresponding \textit{infinite Madelung hierarchy}. We remark that the last formal convergence involves only the indistinguishable property of Bose particles. 

With the aim to face the nonlinear asymptotic mean-field scenario, we propose both a stochastic optimal control formulation of the problem and  a new constrained variational principle starting from the GM Lagrangian equipped with the mean-field density-dependent potential, inspired by \cite{loffredo1996eulerian}, in which the infinitesimal characteristics pair $(\rho,v)$ is forced to satisfy the continuity equation, as is required in the definition of conservative diffusion (\cite{carlen1984conservative}). From both these variational procedures, a nonlinear Madelung system arises involving the gradient of the specific non-local potential. Furthermore, we address the uniqueness problem by introducing the class of admissible solutions to the Madelung equations as those solutions that are also solutions to a Schr\"odinger equation, and we show that the solution to the infinite Madelung hierarchy is the product of copies of the solution to the limit Madelung system. \\

We give our main rigorous convergence results in Section 5. We first prove the weak convergence, for every fixed $ t\in [0,T]$,  of the marginal probability densities and  of the (quantum) current densities (Theorem \ref{cor:Bardos_convergence}). Since both mathematical objects in the previous theorem have a clear quantum mechanics meaning, these convergence results have a quantum mechanical counterpart. 
Successively, we establish the weak convergence of the one-particle process law on the path space.  More precisely, our probabilistic strategy is the following. We consider the $n$-dimensional \textit{conditioned conservative diffusion}, 
which is the solution to the following Brownian motion-driven SDE:
\begin{equation}\label{conditioned_conservativediffusions_intro}
    dX^{N,n}_t=b_{N,n}(X_1(t),X_2(t),\dots,X_n(t))dt + dW^n_t,
\end{equation}
having as drift the original drift conditioned with respect to the first $n$ components of the process, i.e.
$$b_n(X_1,X_2,\dots,X_n,t)=\EE[b_N(X_1,X_2,\dots, X_N,t)|X_1,X_2,\dots,X_n], $$
where $dW^n_t $ is a $nd$-Brownian motion,
and we prove that such a conservative diffusion weakly converges, in the sense of probability measures, to the $n$-dimensional Nelson diffusion:
\begin{equation}\label{limit_conservativediffusion_intro}
    dX^{n}_t=b_{n}(X_1(t),X_2(t),\dots,X_n(t))dt + dW^n_t,
\end{equation}
that is associated with the $n$-copies of the nonlinear Schr\"odinger equation \eqref{nonlin_schro_intro} (Theorem \ref{convergence_path_space}). The proof of the latter theorem uses some recent general results on the description of the stochastic optimal transport problem (\cite{DeVecchi_Rigoni2024})
and some relevant properties of the relative entropy between the involved probability measures, which we combine with the fundamental law of total quantum energy conservation.  A useful representation of the relative entropy between the law of a Nelson diffusion and the Wiener measure in terms of the kinetic energy and the difference of the Boltzmann entropy, evaluated at the initial and final time, is also provided (Lemma~\ref{lem:vu}) in order to complete the proof of our main result.

\medskip

The plan of the present paper is the following.
 The variational formulation of Nelson stochastic mechanics are recalled in Section 2 by providing, in a language more close to the recent stochastic optimal control theory, a revised short description of the GM variational principle of (linear) stochastic mechanics  (\cite{Carlen:162610,GuMo}). A novel variational formulation for the nonlinear (mean-field) stochastic mechanics and an Eulerian 
 variational principle, that allow us to derive the self-interacting Madelung equations, are also proposed.
 In Section 3, the finite and infinite Madelung hierarchies for the marginal infinitesimal characteristics pairs are derived via conditioning, and the formal convergence of the finite Madelung hierarchy to the infinite one is shown.
 The relevant class of conservative diffusions uniquely associated to (linear) Schr\"odinger equations are recalled in Section 4 and the three types of conservative diffusions considered in this paper (Nelson diffusions, conditioned conservative diffusions and self-interacting  Nelson diffusions) are introduced and discussed. Furthermore, some regularity considerations are exploited here in order to guarantee the well-posedness of the osmotic and current velocities and of the related marginal quantities. 
 Our main rigorous convergence results in the above scheme are finally described in the final Section 5.

\section{Variational formulations of stochastic mechanics: the Madelung equations}
Nelson stochastic mechanics was originally formulated by Nelson \cite{NelsonD}  as an alternative formulation of
quantum mechanics.
To the best of our knowledge, up to this day, its relation with quantum mechanics is still an open problem from the physical point of view.
Nevertheless, the theory introduces some important features of quantum mechanics theory into a
well-determined class of diffusion processes whose trajectories could be loosely
interpreted as the trajectories of quantum particles
(\cite{NelsonD},\cite{NelsonQF},\cite{CarlenCD},\cite{Carlen:162610},\cite{carlen1984conservative}). 
For a recent review on Nelson stochastic mechanics see \cite{carlen2014stochastic} and 
for its relation with quantum mechanics see, e.g., \cite{reviewNelson}, \cite{petronimorato}, \cite{pavon2006}, \cite{minellimorato}.

In this section, we recall the original variational formulation of Nelson stochastic mechanics
as a stochastic control problem which corresponds to standard linear Schr\"odinger equations, and we propose two new variational formulations for the asymptotic nonlinear mean-field scenario.

\subsection{Guerra-Morato variational principle}   
We revisit briefly, mainly following \cite{GuMo,Carlen:162610},
the variational formulation of Nelson stochastic mechanics due to Guerra and Morato in 1984 \cite{GuMo}.
This variational problem allows us to derive a nonlinear PDE system,
usually denominated Madelung equations.
These two Madelung equations are formally equivalent to the (linear) $N$-particle Schr\"odinger equation
by a standard well-known transformation (see Remark \ref{standard_transformation}). 

  In the present paper, we discuss neither a possible rigorous formulation of the GM
  variational principle, which will be the focus of future researches, nor the problem of the relation between solutions of the
  Schr\"odinger equation and the ones of Madelung equations (for a discussion regarding this problem see, e.g., \cite{reddiger2023towards} and references therein).

Quantum mechanics can be thought of as a non-commutative
generalization of classical mechanics.
{Nelson stochastic mechanics}
was introduced by Nelson \cite{NelsonD} with the goal 
of reproducing quantum mechanics from
a stochastic, instead of non-commutative, generalization of classical mechanics.
In Nelson theory, the particle kinematic, i.e.\ the path of such a particle,
is postulated to follow a stochastic diffusion process.
The dynamics, i.e. the Newton law, is replaced by a generalized Newton law in terms of forward and backward (Nelson) derivatives
(we refer to \cite{NelsonD} for the details).
In the variational formulation of classical mechanics,
one can replace the dynamics by a \emph{stationary action principle}.
The authors in \cite{GuMo} reformulate the classical principle of stationary action
as a \textit{control problem}. Moreover they generalize such control formulation
to the case of Nelson stochastic mechanics.
Concretely, they show that we can recover the Madelung equations
from a variational principle in a stochastic control context.
This variational procedure consists in considering a diffusion process with drift $b$
and to choose the drift $b$ which is the critical point of a certain cost function.
By critical point (or stationary point) of a cost function, we mean specifically that the variation of the cost function, at the critical point, is equal to zero.

The following re-visitation of the cited variational procedure has,
as the original \cite{GuMo}, a formal character.
Let us consider the following McKean-Vlasov SDE on a finite time horizon $[t_0,t_1]$
\begin{equation}\label{eq:2}
  dQ(t) = b(Q(t),t)dt + dW(t)
  ,
\end{equation}
with initial condition $\Law(Q(t_0))=\rho_0(x)dx$, 
where $\rho_0(x)$ is some fixed, well-behaved, probability density
on $\RR^d$, and where $W$ is an $\RR^d$-valued Brownian motion.
The drift term $b(x,t)$ is considered here formally as a ``regular-enough'' function on $\RR^d\times [t_0,t_1]$ and 
we rewrite it as follows:
\begin{equation}\label{eq:3}
  b(x,t) = v(x,t) + u(x,t),\quad u(x,t)\deq \frac{\nabla\rho(x,t)}{2\rho(x,t)}
  ,
\end{equation}
where the function $v$ is implicitly defined 
in terms of $b$ and $u$ by 
this relation.
In stochastic mechanics $u$ and $v$ are called \textit{osmotic velocity} and \textit{current velocity}, respectively.
       
Assuming the weak solution of~\eqref{eq:2} to be unique,
we denote it by $X(t,b)$ and we denote its law by $\rho(x,t)\dd x=\Law(X(t,b))$.
Moreover, we often do not distinguish between the density $\rho(x,t)$ of the law and the law $\rho(x,t)\dd x$ itself.
Finally, we often write $\rho(t)$ instead of $\rho(\cdot,t)$.

We consider the drift $b$ as a \textit{control} which we want to choose as the critical point of a certain cost function.
For convenience, we assume that the critical point exists and is unique.
In \cite{GuMo} the \textit{cost function}  is chosen as follows:
\begin{align} 
  \label{eq:4}
  B(\Law(X(0;v)),g;{v})
  &=  \EE[\int_0^{t_1}\mathscr L(X(t,b),t;b,Law(X(t;v)))\dd t + S_1(X(t_1,b))] 
    ,
\end{align}
with \textit{deterministic running cost function} $\mathscr L$ given by
\begin{align*}
  \mathscr L(x,t;b,\rho) \deq \tfrac12 (v(x,t)^2 - u(x,t)^2) - V(x),
  ,
\end{align*}
where $u$ is defined as in \eqref{eq:3}
and $S_1$ is a given deterministic \textit{terminal gain function}. 
The gain function $S_1$ is a regular function on $\RR^d$ that  acts as a fixed boundary condition
at the terminal time $t_1$.
\begin{remarks}
  Let us make the following observations.
  \begin{enumerate}[(i)]
  \item In \cite{GuMo} the authors introduce three constants $m,\hbar,\nu$. For simplicity here we take
  $m=\hbar=1$ and $\nu=1/2$. The choice $\nu=1/2$ corresponds to the Schr\"odinger equation case.
  An arbitrary $\nu$ would correspond to a generalization that is outside the scope of the present work.
  \item The running cost function $\mathscr L$ should be considered as a natural stochastic generalization
  of the classical Lagrangian
  \begin{align*}
    \mathscr L = \tfrac12 \dot x^2 - V(x)
    .
  \end{align*}
  Indeed, in the stochastic case it is natural to replace $\dot x^2$, which would lead to infinities,
  with the ``renormalized quantity'' 
  $v^2 - u^2$ (cf.\ \cite{GuMo} and \cite{carlen1984conservative}).
  Similar Lagrangian functional are common in quantum mechanics and quantum field theory
  (cf.\ e.g.\ \cite{grodnik1970representations}
  for a similar Lagrangian density
  in the context of current algebras).
  \item Finally, we stress that the sign in front of $u^2$, in the deterministic running cost function $\mathscr L$,
  says that the osmotic velocity plays the role of an additional (quantum) potential.  
  \end{enumerate}
\end{remarks}
Following \cite{GuMo}, we first note that a simple integration by parts with respect to the measure
$\rho(x,t)\dd x$ gives
\begin{align*}
  \EE[\mathscr L(x,t;\rho,b)] =   \EE[\mathscr L_+(x,t;b)] ,
\end{align*}
with
\begin{align*}
  \mathscr L_+(x,t;b) \deq \tfrac12 b(x,t)^2 + \tfrac12(\nabla b)(x,t) -V(x)
  .
\end{align*}
It is useful to introduce a cost function for the process starting from $x\in\RR^d$
in terms of $\mathscr L_+$:
\begin{align} 
  \label{eq:5}
  J_+(x,t;{b})
  &\deq  \EE[\int_{t}^{t_1}\mathscr L_+(X(t,b),t;b)\dd t + S_1(X(t_1,b))| X(t,b)=x] 
    ,
\end{align}
where, for brevity, we have suppressed the dependence of $J_+$ on the final time $t_1$ and on the
terminal gain function $S_1$.

We remark that we could also have defined a cost function $J$ with $\mathscr L_+$
replaced by the original $\mathscr L$.
We shall come back to this observation
in Remark~\ref{rem:Lplus} below.
Moreover, note that the cost function $B$ can be expressed in terms of $J_+$ by integrating it over the initial distribution:
\begin{equation}\label{eq:6}
  B(\rho_0,g;{v}) =   \int_{\RR^d} J_+(x,t;g,{v})\rho_0(x)\dd x
  .
\end{equation}

In \cite[(109)]{GuMo}, the authors note that $J_+$ defined in \eqref{eq:5} satisfies the following transport equation
\begin{align*}
  \partial_t J_+(x,t;b)  + (b(x,t)\cdot \nabla) J_+(x,t;v) + \tfrac12 (\Delta J_+)(x,t;v)
  = \mathscr L_+(x,t;\bar{b})
  ,
\end{align*}
where we have written $\rho(x,t)$ as $\Law(X(t,b))$ hence stressing the fact
that the probability density $\rho(x,t)$ of the solution $X(t,b)$ depends in fact on $b$.

Since we have assumed that the critical point is unique,  let us denote by $\bar b$ the critical control $b$
and by $S(x,t)\deq J_+(x,t;\bar{b})$ the \textit{value function},
i.e. the value of $J_+$ at the critical point $b\equiv\bar b$.
Then, taking $b\equiv \bar{b}$ in the formula above, we obtain 
the following  \textit{Hamilton-Jacobi-Bellman} (HJB) equation:
\begin{equation}
  \label{eq:7}
  \partial_t S(x,t)  + (\bar b(x,t)\cdot \nabla) S(x,t) + \tfrac12 (\Delta S)(x,t)
  = \mathscr L_+(x,t;\bar{b},\Law(X(t;\bar{b})))
  .
\end{equation}
\begin{remark}\label{rem:Lplus}
  Let us now come back to the role of $\mathscr L_+$ vs $\mathscr L$.
The HJB equation above is formal in the sense that it should be intended in a weak sense,
that is, we should consider it after integrating with respect to the measure $\rho(x,t)\dd x$.
Because of this integration, we can take as right-hand side either $\mathscr L_+$ or $\mathscr L$.
Taking $\mathscr L_+$, as in \cite{GuMo}, leads to the Madelung equations in the usual form.
The advantage of employing $\mathscr L_+$, in place of $\mathscr L$, is that $\mathscr L_+$
 depends only on $b$ whereas $\mathscr L$ depends on $b$ as well as on the law $\rho$.
Because of this, the variation of $J_+$ can be computed more easily. 
\end{remark}  
Concretely, in \cite{GuMo}, it is shown that the variation of $J_+$
when we change $b\rightarrow b+\delta b$ can be expressed as follows
(cf. \cite[(114)]{GuMo})
\begin{equation}
  \label{eq:8}
  \delta B(t_0,t_1) = \int_{t_0}^{t_1} \int (v - \nabla J_+)\cdot \delta b(x,t) \rho(x,t) \dd x \dd t
  ,
\end{equation}
where  $v(x,t) = b(x,t) - u(x,t)$ by \eqref{eq:3}.
Since the variation $\delta b$ is generic, imposing the critical condition $\delta B\equiv 0$,
gives the following \textit{Hamilton-Jacobi condition}
for the critical value $\bar b$ of $b$
(cf.\ \cite[(117)]{GuMo}):
\begin{equation}
  \label{eq:9}
  \bar b(x,t) = \bar u(x,t) + \nabla S(x,t)
  ,
  \quad \bar{u}\deq \frac{\nabla\bar\rho(x,t)}{2\bar{\rho}(x,t)}
  ,
\end{equation}
where $\bar\rho(x,t)\deq \text{Law}(X(t,\bar{b}))$ denotes the law of the solution $X(t,\bar{b})$
of the SDE \eqref{eq:2} with drift taken to be the critical one.
Note that this Hamilton-Jacobi condition is an irrotational condition. Using this condition in the Hamilton-Jacobi-Bellman equation \eqref{eq:7}, we obtain
the following Madelung equation (cf.\ \cite[(122)]{GuMo})
\begin{align}
  \label{eq:10}
  \partial_t S(x,t) + \tfrac12 (\nabla S)^2(x,t) -\tfrac12 [u^2(x,t) + (\nabla \cdot u)(x,t)] + V(x) = 0
  .
\end{align}

Let us now, for convenience, redefine 
$v$ to be
\begin{align*}
  v(x,t) \deq \bar{b}(x,t) - \bar{u}(x,t),
\end{align*}
in such a way to remove the overline $\bar{b}$
in the final formulas.
Combining \eqref{eq:9} and \eqref{eq:10} and with this notation for $v$ we derive that $v$
satisfies the equation
\begin{align}\label{eq:11} 
  \partial_t v(x,t) + (v(x,t)\cdot\nabla)v(x,t) -\frac12 \nabla [u^2(x,t) + (\nabla \cdot u)(x,t)]  +\nabla V(x) =0.
\end{align}

Finally let us note that $\rho(x,t)$, being the density of the probability distribution of the solution $X(t,b)$ of~\eqref{eq:2},
satisfies the Fokker-Planck equation
\begin{align*}
  \partial_t\rho(x,t) + \nabla\cdot [(\rho(x,t) b(x,t))] -\tfrac12(\Delta\rho)(x,t) =0
  .
\end{align*}
By employing the definition of $v$ in \eqref{eq:3}, we obtain that $\rho(x,t)$ 
also satisfies the
\textit{continuity equation}
\begin{equation}
  \label{eq:12}
  \partial_t \rho(x,t) + \nabla\cdot [(\rho(x,t) v(x,t))] =0
  .
\end{equation}
The nonlinear system given by \eqref{eq:12} and \eqref{eq:11}:
\begin{equation}\label{eq:2000}
  \begin{cases}
  &\partial_t \rho(x,t) + \nabla\cdot [(\rho(x,t) v(x,t))] =0\\
  &\partial_t v(x,t) + (v(x,t)\cdot\nabla)v(x,t) -\frac12 \nabla [u^2(x,t) + (\nabla \cdot u)(x,t)]  +\nabla V(x) =0
  ,
  \end{cases}
\end{equation}
is called the \textit{Madelung equations}.
In the following, we shall often refer to \eqref{eq:12} as the continuity equation and to \eqref{eq:11}
as the second Madelung equation.
\begin{remark}\label{standard_transformation}
 The Madelung equations can be derived from the Schr\"odinger equation via the change of variables
(cf. \cite{GuMo})
\begin{equation}
  \label{eq:13}
  \Psi(x,t)=\sqrt{\rho(x,t)}\exp{iS(x,t)},\quad  v(x,t)=\nabla S(x,t)
\end{equation}
where $\rho(x,t)=|\Psi(x,t)|^2$.
This is the well-known link between the linear Schr\"odinger equation and Nelson stochastic mechanics. When $\Psi$ is regular and without nodes (e.g. $\Psi \in C^1$, $\Psi>0$), then one can derive from $\Psi$ the couple $(\rho, S).$ In general, starting from the Madelung variables $(\rho, S)$ an additional quantization condition of the circulation must be imposed to obtain $\Psi$ (see \cite{reddiger2023towards} for an articulated discussion on the topic).
\end{remark}

\begin{remark}
  Note that the critical drift $\bar{b}$ 
  comes from a solution to the Madelung equations which are a nonlinear system of equations.
  Hence it can be thought of as a nonlinear function of $\rho$, i.e. as a nonlinear function depending on the law of
  the process (if we explicitly take into account the dependence on the initial measure).
  Therefore, when we consider the SDE \eqref{eq:6} with the critical drift $\bar{b}$,
  we are in fact considering a nonlinear diffusion of McKean-Vlasov type.
  For this reason, GM variational principle can be considered naturally
  in the context of critical control problems of McKean-Vlasov type.
  Optimal control problems for McKean-Vlasov diffusions
  are considered e.g.\ in \cite{DeVecchi_Rigoni2024,bao2023ergodic,carmona2018probabilistic} and
  we  would like to mention also \cite{von2012optimal} for a  context similar to ours.
  We point out that the GM variational principle is a \textit{critical} control problem
  instead of an \textit{optimal} control problems, because we are considering
  a cost-function which is non-convex and we are looking for a critical-point (i.e.\ a stationary point) of such a cost function
  instead of an optimal point (i.e. a minimum).
\end{remark} 

\subsection{Variational  formulation of non-linear stochastic mechanics: the self-interacting Madelung equation}
 
From its beginning Nelson stochastic mechanics is related to linear Schr\"odinger equations.
If we want a stochastic mechanics which enjoys the same relation with a nonlinear Schr\"odinger equation,
we expect in general to have to develop an extension of Nelson stochastic mechanics.
It turns out that in the case of \eqref{nonlin_schro_intro},
the extension we need to come up with is particularly simple.
In fact, the change of variables described in \eqref{eq:13},
applied to the nonlinear Schr\"odinger equation \eqref{nonlin_schro_intro},
leads to Madelung equations of the same form as \eqref{eq:12} and \eqref{eq:11} where
the potential $V$ is replaced by the nonlinear potential
$V\ast\rho(t)$ where $\rho(t)=|\psi(t)|^2$ denotes as before the density of the process at time $t$,
with $\psi(t)$ solution of the nonlinear Schr\"odinger equation \eqref{nonlin_schro_intro}.
In the context of Nelson stochastic mechanics, we refer to this case as the
\textit{self-interacting case}.
We use this terminology because, regardless of the potential being linear or nonlinear,
the kinematics of Nelson stochastic mechanics is intrinsically described by a nonlinear (i.e. McKean-Vlasov type) diffusion.

\paragraph{Critical control approach.}
We formulate the \textit{self-interacting case} of Nelson stochastic mechanics,
similarly to the non self-interacting one, as a controlled McKean-Vlasov problem.
We start with a McKean-Vlasov SDE as in~\eqref{eq:2} and
we denote the associated cost function by $B^\SE$, where the superscript $^\SE$ stays for self-interacting.
Explicitly, in this case we take
\begin{align*}
  B^\SE(\Law(X(0,b));{v})
  &=  \EE[\int_0^{t_1}\mathscr L^\SE(X(t,b),t;b, \Law(X(t,b)))\dd t + S_1^\SE(X(t_1,b)] 
    ,
\end{align*}
with $\mathscr L^\SE$ given by
\begin{align*}
  \mathscr L^\SE(x,t;b,\rho(t)) \deq \tfrac12 (v(x,t)^2 - u(x,t)^2) - V\ast\rho(x,t)
  ,\quad x\in\RR^d
\end{align*}
where, as before, the control $b:\RR^d\times \RR\rightarrow\RR^d$ is assumed to be a regular function, $u$ is given in \eqref{eq:3}, 
and the terminal gain function $S^\SE_1$ is some regular deterministic function.
Note that in $\mathscr L^\SE$ we now have the nonlinear potential $V\ast\rho$.

As in the non self-interacting case, we introduce the modified running cost function
\begin{align*}
  \mathscr L^\SE_+(x,t;b,\rho(t)) \deq \tfrac12 b(x,t)^2 + \tfrac12(\nabla b)(x,t)
  -V\ast\rho(x,t)
  ,
\end{align*}
and define the cost function for the process starting from a point $x$ as
\begin{align*}
    J^\SE_+(x,t;{b})
  &\deq  \EE[\int_{t}^{t_1}
  \mathscr L^\SE_+(X,t;b,\rho(X(t;b),t))\dd t + S_1(X(t_1;b))| X(t;b)=x] 
    ,
\end{align*}
with the same notation as in~\eqref{eq:5}.
In order to obtain the HJB equation, we
continue to work formally without any modification  also in the case of a self-interacting potential
$V\ast\rho$. We can do this since \eqref{eq:7}, formally, holds regardless of which
running cost function we use.
Therefore in this nonlinear case we derive the following HJB equation
\begin{equation}
\label{eq:14}
  \partial_t S^\SE(x,t)  + (b(x,t)\cdot \nabla) S^\SE(x,t) + \tfrac12 (\Delta S^\SE)(x,t)
  = \mathscr L^\SE_+(x,t;\bar{v},\Law(X(t;\bar{v})))
  ,
\end{equation}
where, consistently with the notation above, we have denoted by $S^\SE$
the value function associated to $J^\SE_+$.

Note that, contrary to the situation considered in the previous subsection, the cost function $\mathcal L_+^{\text{se}}$ still depends on the law of the process 
$\rho(t)$ via the self-interacting potential $V\ast\rho$ which depends on $\rho$.
To eliminate this dependence of the potential on $\rho$, and to obtain a Hamilton-Jacobi condition similar to \eqref{eq:9}, we cast this
self-interacting control problem into another control problem which is
\textit{not} self-interacting, i.e. we reduce the self-interacting case to
the situation considered in \cite{GuMo} and recalled in the previous subsection.

The new control problem consists of two independent identically distributed processes
\begin{equation}\label{eq:15}
  \begin{cases}
    dQ_1(t) &=  b(Q_1(t),t) + dW_1(t),\\ 
    dQ_2(t) &= b(Q_2(t),t) + dW_2(t)
    ,
  \end{cases}
\end{equation}
subject to identical initial conditions $\Law(Q_1(t_0))=\Law(Q_2(t_0))=\rho_0(x)\dd x$.
Note that in both SDEs above the drift term is the same and can be written, as before, as in 
\eqref{eq:3} where $\rho(x,t)$ is the density of
both $Q_1(t)$ and $Q_2(t)$, because they are identically distributed.
If we denote the joint distribution by $\rho_2(x,y,t)$ we have, by independence, 
that $\rho_2(x,y,t)=\rho(x,t)\rho(y,t)$.
At this point we introduce the following running cost function
\begin{align}\label{eq:16}
  \mathscr L_2(x,y,t;b,\rho) \deq \tfrac12 (v(x,t)^2 -  u(x,t)^2) - V(x-y)
  .
\end{align}
This Lagrangian is a special case of \eqref{eq:4} with ``doubled'' number of dimensions,
i.e. here we replace $\RR^d$ with $\RR^{2d}$.
It is easy to see that, formally,
this problem is equivalent to the self-interacting problem described above.
Indeed by a simple computation 
\begin{equation}
\label{eq:17}
  \EE_2[   \mathscr L_2(X_1,X_2,t;b,\rho) ]
  =
  \EE[ \mathscr L^{\text{se}}(X_1,t;b,\rho)]
  ,
\end{equation}
where $\EE_2$ denotes the expectations with respect to the joint law
of the identically distributed processes
$X_1,X_2$ that are solutions to~\eqref{eq:15}.

Since this new control problem is a special case of the one considered in \cite{GuMo},
we can follow the same steps as in the previous subsection.
In particular, we define the cost function $B$ similarly to \eqref{eq:6}
in terms of a $J_+$ similar to \eqref{eq:5},
but where the potential  now depends on the two processes $X_1$, $X_2$ as in \eqref{eq:17},
and the expectations $\EE$ is replaced with the joint expectation $\EE_2$.
Moreover we can employ formula~\eqref{eq:8} which, in the special case at hand, becomes
\begin{equation}
\label{eq:18}
\delta B(t_0,t_1) = \int_{t_0}^{t_1} \int (v(x,t) - \nabla_x J_{+}(x,y,t))\cdot \delta b(x,t)
\rho_2(x,y,t) \dd x\dd y \dd t
  .
\end{equation}
By independence of the two processes we can also write
\begin{equation}\label{eq:19}
\delta B(t_0,t_1)
= \int_{t_0}^{t_1} \int \Big(v(x,t) - \nabla_x \big[\int J_{+}(x,y,t)\rho(y,t)\dd y\big]\Big)
\cdot \delta b(x,t)
\rho(x,t) \dd x \dd t
  .
\end{equation}
Note now that
\begin{align*}
  \int J_{+}(x,y,t)\rho(y,t)\dd y = J^\SE_+(x,t)
  .
\end{align*}
This follows by the definition of $J_{+}$ in a similar fashion to \eqref{eq:17}, which implies the following Hamilton-Jacobi condition:
\begin{equation}
\label{eq:20}
  \bar v(x,t) = \nabla S^\SE (x,t)
  ,
\end{equation}
where $S^\SE$ denotes the value function associated to $J^\SE_+$ as above.

Putting together condition \eqref{eq:20} with the HJB equation \eqref{eq:14}, 
we get the following second Madelung equation

\begin{align}\label{eq:21} 
  \partial_t v(x,t) + (v(x,t)\cdot\nabla)v(x,t) -\frac12 \nabla [u^2(x,t) + (\nabla \cdot u)(x,t)]  +2(V\ast\rho)(x) =0,
\end{align}
where we have use $v$ and not $\bar v$ to conform with
the notation in the following and we have removed the suffix $\SE$ for brevity.
As expected, this Madelung equation for the self-interacting case is obtained
formally from the Madelung equation of the non-self-interacting case~\eqref{eq:33}
simply by replacing the potential $V$ with the self-interacting potential $V\ast\rho$.
We also note that the continuity equation~\eqref{eq:12} for the non-self-interacting case
remains unchanged when we pass to the self-interacting case since it does not depend
on the potential $V$, that is, it does not depend on the dynamics but only on the kinematics
of the process.

\paragraph{Eulerian approach.}
The GM variational principle described in the previous section belongs to the class of Eulerian approaches to variational principles of classical mechanics, where the fluid is represented in terms of a velocity field and a thermodynamic quantity as the density. This principle provides only the irrotational drift coefficient, consistently with the standard quantum mechanics theory. In \cite{loffredo1996eulerian} some generalizations of the Eulerian stochastic variational principles are proposed in order to obtain also rotational solutions to Madelung equations by means of suitable Lagrangian multipliers. Here we face the case of nonlinear mean-field-type potential within the framework of standard irrotational solution of a stochastic variational principle.

More specifically, we aim to recover the Madelung equations for the nonlinear Schr\"odinger equation
via a novel variational approach.
Inspired by \cite{loffredo1996eulerian},
we formulate an enlarged stochastic variational principle of Eulerian type starting from the mean action of
Nelson stochastic mechanics for the nonlinear density-dependent potential.
Since for the class of conservative diffusions the continuity equation has to be always satisfied
(see Definition \ref{conservativediffusiondefinition}),
we consider a constraint variational  principle according with the following action on $\R^d$
\begin{equation*}
  A(\rho,v,\lambda)=\int_0^T \Big\{\int [\tfrac{1}{2}(v^2-u^2)(x,t)-V\star \rho(x,t)]\rho(x,t)dx+\int \lambda(x,t)[\partial_t \rho+\nabla \cdot (v\rho)] (x,t)dx\Big\}dt,  
\end{equation*}
where $\lambda $ is a Lagrangian multiplier and $u$ is as in \eqref{eq:3}.

The variation with respect to $\lambda$ gives at once the continuity equation \eqref{eq:12}, while
the variation with respect to the current velocity $v$  at first order in $\delta v $ provides the following relation:
\begin{equation}\label{eq:22} 
  \int_0^T \int \Big\{ v - \nabla\lambda(x,t)\Big\}\cdot \delta v\rho(x,t)dxdt=0,   
\end{equation}
which, by integration by parts, yields the well-known irrotational Hamilton-Jacobi condition as in \eqref{eq:9}.

The variation with respect to the probability density $\rho$ gives the following.
\begin{equation}\label{eq:24}
  \int_0^T \Big\{\int \tfrac{1}{2}(v^2-u^2)\delta\rho - (u\delta u)\rho
  -  [ (V\ast\rho)\delta\rho + (V\ast\delta\rho)\rho]  \Big\}dx 
  + \int \lambda [\partial_t \delta \rho+\nabla \cdot (v\delta \rho)]dx \Big\}dt=0
  ,
\end{equation}
where the notation $\delta u=u(\rho+\delta\rho)-u(\rho)$ has been introduced.
By definition of $u$ in \eqref{eq:3}, to first order in $\delta\rho$, we have
$\delta u =\frac12 \nabla\big(\frac{\delta\rho}{\rho}\big)$.
Therefore, an integration by parts gives
\begin{equation}\label{eq:25}
  \int u\delta u\rho \dd x
  = -\int (\tfrac12\nabla\cdot u + u^2)\delta\rho \dd x
  .
\end{equation}
Moreover, by symmetry of $V$ under the inversion transformation $x\rightarrow-x$, a simple computation shows that,
to first order in $\delta\rho$, one gets
\begin{equation}\label{eq:26}
  \int [ (V\ast\rho)\delta\rho + (V\ast\delta\rho)\rho] \dd x
  = 2\int (V\ast\rho)\delta\rho\dd x
  .
\end{equation}
Integrating by parts the last addendum in \eqref{eq:24} we obtain
\begin{equation}\label{eq:27}
  \int_0^T \int \lambda [\partial_t \delta\rho+\nabla \cdot (v\delta\rho)]\dd x\dd t
  = - \int_0^T \int \Big(\partial_t \lambda + (v \cdot \nabla \lambda\Big) \delta\rho\dd x\dd t.
\end{equation}
Using \eqref{eq:25}, \eqref{eq:26}, and \eqref{eq:27} we can finally rewrite \eqref{eq:24} as follows
\begin{align*}
  \int_0^T \int\Big\{\tfrac{1}{2}(v^2+u^2+\nabla\cdot u )-2(V\ast\rho)
  -\Big(\partial_t \lambda + (v \cdot \nabla) \lambda\Big)\Big\}\delta\rho
  \dd x\dd t
  =0
  , 
\end{align*}
which gives in terms of $v=\nabla \lambda=\nabla S$ the following Madelung equations:
\begin{equation}\label{eq:29}
  \begin{cases}
    &\partial_t\rho +\mathrm{div}(\rho v)=0\\
    &\partial_t v + (v\cdot\nabla)v -\tfrac12 \nabla [{u^2+\nabla \cdot u}]
      +2 \nabla(V\ast\rho)=0
      .
  \end{cases}
\end{equation}

\section{Madelung hierarchies}\label{sec:madelung}
 We introduce  a new stochastic procedure for reducing the $N$-particle conservative diffusion system to an $n$-particle one given by  a fixed $n$-``marginal diffusions'', taking inspiration from the quantum approach (\cite{Bardos}).
This program, which we start in this section
by deriving both a Madelung hierarchy for the $N$ interacting diffusions system and a related infinite Madelung hierarchy, will conclude in Section 5. 
\begin{notation}
  It will be convenient to introduce the following notation.
  \begin{enumerate}[(i)]
  \item A coordinate vector in $\RR^{Nd}$ will be denoted by $x^N$.
    The first $n$ coordinate of $x^N$ will be denoted by $x_1^n$
    or simply by $x^n$
    and the remaining coordinates by $x_{n+1}^N$.
    
  \item Given a generic vector $v\in\RR^{Nd}$, we shall denote its projection onto the first $n$ components by $[v]_1^n\in\RR^{nd}$.

  \item Random-variables will be denoted with capital letters.
    For example, a random-variable with values in $\RR^{Nd}$ will be denoted by $X^N$,
    with $X_1^n$ its first $n$ components and $X_{n+1}^N$ its remaining components.
  \end{enumerate}
\end{notation}

\subsection{Conditional expectations}
In parallel to what is done in quantum mechanics passing from
the density matrix to the marginal density matrices,
given the Madelung variables $(\rho_N,v_N)$ for a system of $N$-particles,
we want to introduce the \textit{conditioned Madelung variables}. 
By denoting with
$[v_N(x^N,t)]_1^n$ the vector with components the first $n$ components of the $N$ component vector
$v_N(x^N,t)$, we provide the following definitions. 
\begin{definition}\label{marginal_Madelung_variables}
  We define the \textit{marginal Madelung} variables as follows
  \begin{equation}\label{eq:30}
    \begin{aligned}
      \rho_{N,n}(x_1^n,t)
      &\deq \int \rho_N(x_1^n,x_{n+1}^N,t)\dd x_{n+1}^N \\
      u_{N,n}(x^1_n,t)
      &\deq \frac{1}{ \rho_{N,n}(x_1^n;t)  }
        \int  [u_N(x_1^n, x_{n+1}^N,t) ]_1^n \rho_N(x_1^n,x_{n+1}^N,t) \dd x_{n+1}^N
      \\
      v_{N,n}(x_1^n,t)
      &= \frac{1}{ \rho_{N,n}(x_1^n;t) }
        \int
        [v_N(x_1^n,x_{n+1}^N,t)]_1^n \rho_{N}(x_1^N,t) \dd x_{n+1}^N
        = \frac{j_{N,n}(x_1^n,t)}{\rho_{N,n}(x_1^n,t)}
        ,
    \end{aligned}
  \end{equation}
  where
   \begin{equation}
    \label{eq:31}
    j_{N,n}(x_1^n,t)
    \deq \rho_{N,n}(x_1^n,t) v_{N,n}(x_1^n,t)
    = \int [j_N(x_1^n,x_{n+1}^N)]_1^n \dd x_{n+1}^N
    .
  \end{equation}    
\end{definition}
\begin{remark}
    Let us motivate the marginal Madelung variables as given above.  
  Let $X^N$ be a random variable with values in $\RR^{Nd},$ which is distributed according to
  a probability measure $\dd\mu(x^N)= \rho_N(x^N,t)\dd x^N$, which is absolutely continuous with density $\rho_N$
  with respect to the Lebesgue measure $\dd x^N$ in $\RR^{Nd}$, for any fixed $t$.
  Then, for an integrable function $F_N$ we consider the following conditional expectation:
  \begin{equation}
    \label{eq:32}
    \begin{aligned}
      F_{N,n}(X_1,\dots,X_n,t)
      &\deq \E[ F_N(X_N,t) | X^1, \dots, X^n ]\\
      &=\frac{1}{\rho_{N,n}(X^1,\dots,X^n,t)}
        \int F(X_1,\dots,X_n, x_{n+1},\dots,x_N,t) \times\\
      &\hspace{14em}\times\rho_N(X_1,\dots,X_n,x_{n+1},\dots,x_N,t) \dd x_{n+1}\cdots \dd x_N.
    \end{aligned}
  \end{equation}
  Passing from $F_N$ to $F_{N,n}$ corresponds to ``integrating away''
  the variables from $n+1$ to $N$.
  This justifies the form of $u_{N,n}$ and $v_{N,n}$ in the definition above.
  We shall come back to this point when we introduce the \textit{conditional Nelson diffusions}
  (cf.\ \eqref{eq:45}).
\end{remark}

\subsection{Finite and infinite Madelung hierarchies}
Starting from the conditioned infinitesimal characteristics $(\rho_{N,n},v_{N,n})$
introduced in Definition \ref{marginal_Madelung_variables} above,
we now provide a ``conditioned version'' of Madelung equations \eqref{eq:2000} and
we shall call these equations the \emph{finite Madelung hierarchy}.

\begin{proposition}
The conditioned infinitesimal characteristics $(\rho_{N,n},v_{N,n})$,
given in Definition~\ref{marginal_Madelung_variables} above,
satisfy the following \textit{finite Madelung hierarchy}:
\begin{equation}\label{eq:61}
  \begin{aligned}
    \partial_t\rho_{N,n}(X_1^n,t)
    = &\;\mathrm{div}_1^n (\rho_{N,n} v_{N,n})(X_1^n;t),
    \\
    \partial_t v_{N,n}(X_1^n,t)
    = &  -\int_{\R^{N-n}}([v_N]_1^n\cdot\nabla)_1^n [v_N]_1^n
        \frac{\rho_N(x_1^n,x_{n+1}^N,t)}{\rho_{N,n}(x_1^n,t)}dx^{n+1}_N\\
      & - \int \Big[ [v_N]_1^n \Big( ({\div}_1^n[v_N]_1^n) + 2 [v_N]_1^n\cdot [u_N]_1^n\Big)
      \frac{\rho_N(x_1^n,x_{n+1}^N,t)}{\rho_{N,n}(x_1^n,t)}\dd x_{n+1}^N + \\
      &\hspace{15em}  + [v_{N,n}]_1^n \Big( {\div}_1^nv_{N,n} + 2 v_{N,n}\cdot u_{N,n}\Big)+\\
      & + \int_{\R^{N-n}}\tfrac12 [\nabla]_1^n \Big[ u^2_{N}+\mathrm{div}\cdot (u_N) \Big]
        \frac{\rho_N(x_1^n,x_{n+1}^N,t)}{\rho_{N,n}(x_1^n,t)}dx^{n+1}_N +\\
      &\hspace{1cm}- \frac{1}{N} \sum_{1\le j<k\le n}  \int [\nabla]_1^n V(x^j-x^k) 
        \frac{\rho_N(x_1^n,x_{n+1}^N;t)}{\rho_{N,n}(x_1^n;t)} \dd x_{n+1}^N  +\\
      &\hspace{2cm}- \frac{N-n}{N} \sum_{1\le j\le n }  \int [\nabla]_1^n V(x^j-z) 
        \frac{\rho_{n+1,N}(x_1^n,z,t)}{\rho_{N,n}(x_1^n,t)} \dd z
        .
  \end{aligned}
\end{equation}
\end{proposition} 
\begin{proof}
    With the aim to obtain the finite Madelung hierarchy, we start by computing the time evolution
for the conditioned variables (Definition \ref{marginal_Madelung_variables}).
From now on all the equations are supposed to be considered in a weak sense. 
From the first equation in Definition \ref{marginal_Madelung_variables}, by using the continuity equation in \eqref{eq:12},  we get
\begin{align*}
  \int_{\RR^{nd}} \phi(x_1^n) \partial_t\rho_{N,n}(x_1^n,t)\dd x_1^n
  &= \int_{\RR^{Nd}} \phi(x_1^n) \partial_t \rho_N(x_1^n,x_{n+1}^N,t)\dd x^N\\
  &= -\int_{\RR^{Nd}} \phi(x_1^n) \text{div} j_N(x_1^N,t)\dd x^N\\
  &= \int_{\RR^{Nd}} \nabla\phi(x_1^n) \cdot j_N(x_1^N,t)\dd x^N\\
  &= \int_{\RR^{Nd}} [\nabla\phi(x_1^n)]_1^n \cdot [j_N]_1^n(x_1^N,t) \dd x^N\\
  &= \int_{\RR^{Nd}} [\nabla\phi(x_1^n)]_1^n \cdot j_{N,n}(x_1^N,t) \dd x_1^n
  \\
  &= -\int_{\RR^{Nd}} \phi(x_1^n) \text{div} j_{N,n}(x_1^N,t) \dd x_1^n
    ,
\end{align*}
which gives that  $\rho_{N,n}$ solves weakly the continuity equation:
\begin{equation}\label{eq:33}
  \partial_t \rho_{N,n}
  = -\text{div}_1^n j_{N,n}
  .
\end{equation}
From the second equation in Definition \ref{marginal_Madelung_variables} and
by introducing the second Madelung equation in \eqref{eq:11} we get:
\begin{equation}\label{eq:59}
  \begin{aligned}
    \partial_t v_{N,n}
    &= \partial_t
      \int [v_N]_1^n \frac{\rho_N}{\rho_{N,n}} \dd x_{n+1}^N
    \\
    &= \int [\dot v_N]_1^n \frac{\rho_N}{\rho_{N,n}} \dd x_{n+1}^N
      + \int [v_N]_1^n \frac{ \dot\rho_N \rho_{N,n} - \rho_N\dot\rho_{N,n} }{(\rho_{N,n})^2}
    \\
    &= \int \Big( -(v_N\cdot\nabla)[v_N]_1^n
      +\tfrac12 [\nabla]_1^n [u_N^2+\nabla \cdot u_N] -[\nabla]_1^n V \Big)
      \frac{\rho_N}{\rho_{N,n}} \dd x_{n+1}^N + \\
    &\hspace{25em}+ \int [v_N]_1^n \frac{ \dot\rho_N \rho_{N,n} - \rho_N\dot\rho_{N,n} }{(\rho_{N,n})^2}
      .
  \end{aligned}
\end{equation}
To simplify the previous equation further, as in the quantum mechanical approach, we restrict our attention
to a two-body potential $V_N$, that is of the form:
\begin{align*}
  V_N(x_N) = \frac{1}{N} \sum_{1\le j<k\le N} V_N(x_j-x_k)
  .
\end{align*}
For this potential we have
\begin{align*}
  \partial_\ell V_N(X_N)
  &= \frac{1}{N} \sum_{1\le j<k\le N} \partial_\ell V_N(x_j-x_k)\\
  &= \frac{1}{N} \sum_{1\le j<k\le N}  V_N'(x_j-x_k) (\delta_{j\ell} - \delta_{k\ell})
    ,
\end{align*}
hence the gradient in the first $n$ directions can be written as
\begin{align*}
  [\nabla]_1^n V_N(X_N)
  &= \frac{1}{N} \sum_{\substack{ 1\le j\le n \\ n<k\le N }} [\nabla]_1^n V_N(x_j-x_k)
  .
\end{align*}
Using this fact, let us consider the term
\begin{align*}
  &\int [\nabla]_1^n V_N(x_1^n,x_{n+1}^N) 
    \frac{\rho_N(x_1^n,x_{n+1}^N,t)}{\rho_{N,n}(x_1^n,t)} \dd x_{n+1}^N = \\
  &= \frac{1}{N}\sum_{1\le j<k\le N}  \int [\nabla]_1^n V_N(x^j-x^k) 
    \frac{\rho_N(x_1^n,x_{n+1}^N,t)}{\rho_{N,n}(x_1^n,t)} \dd x_{n+1}^N  \\
  &= \frac{1}{N} \sum_{1\le j<k\le n}  \int [\nabla]_1^n V_N(x^j-x^k) 
    \frac{\rho_N(x_1^n,x_{n+1}^N,t)}{\rho_{N,n}(x_1^n,t)} \dd x_{n+1}^N  \\
  &\quad+ \frac{1}{N} \sum_{\substack{\phantom{n+}1\le j\le n \\ n+1\le k\le N}}  \int [\nabla]_1^n V_N(x^j-x^k) 
  \frac{\rho_N(x_1^n,x_{n+1}^N,t)}{\rho_{N,n}(x_1^n,t)} \dd x_{n+1}^N
  .
\end{align*}
Finally by taking advantage of the symmetry under permutation of $\rho_N$ we obtain
\begin{equation}\label{eq:60}
  \begin{aligned}
    &\int [\nabla]_1^n V_N(x_1^n,x_{n+1}^N) 
      \frac{\rho_N(x_1^n,x_{n+1}^N,t)}{\rho_{N,n}(x_1^n,t)} \dd x_{n+1}^N = \\
    &= \frac{1}{N} \sum_{1\le j<k\le n}  \int [\nabla]_1^n V_N(x_j-x_k) 
      \frac{\rho_N(x_1^n,x_{n+1}^N,t)}{\rho_{N,n}(x_1^n,t)} \dd x_{n+1}^N  \\
    &\quad+ \frac{N-n}{N} \sum_{1\le j\le n }  \int [\nabla]_1^n V_N(x_j-z) 
      \frac{\rho_{n+1,N}(x_1^n,z,t)}{\rho_{N,n}(x_1^n,t)} \dd z
      .
  \end{aligned}
\end{equation}
Let us now consider the last term in \eqref{eq:59}.
Indeed, by using the continuity equation, we can rewrite it as follows
\begin{equation}
  \label{eq:34}
  \begin{aligned}
    &\int [v_N]_1^n \frac{ \dot\rho_N \rho_{N,n} - \rho_N\dot\rho_{N,n} }{(\rho_{N,n})^2} \dd x_{n+1}^N
    \\
    &=\int [v_N]_1^n
      \Big[ \frac{\partial_t \rho_N}{\rho_{N}}  - \frac{\partial_t \rho_{N,n}}{\rho_{N,n}} \Big]
      \frac{\rho_N}{\rho_{N,n} }\dd x_{n+1}^N
    \\
    &=\int [v_N]_1^n
      \Big[ \frac{-\text{div}_1^N(v_N \rho_N)}{\rho_{N}}  + \frac{\text{div}_1^n(v_{N,n} \rho_{N,n})}{\rho_{N,n}}\Big]
      \frac{\rho_N}{\rho_{N,n} }\dd x_{n+1}^N
    \\
    &=\int [v_N]_1^n \Big[- \frac{\text{div}_1^n(v_N \rho_N)}{\rho_{N}}
      + \frac{\text{div}_1^n(v_{N,n} \rho_{N,n})}{\rho_{N,n}}\Big]\frac{\rho_N}{\rho_{N,n} }\dd x_{n+1}^N -\\
    &\hspace{23em} -\int [v_N]_1^n \Big[ \frac{\text{div}_{n+1}^N(v_N \rho_N)}{\rho_{N}}\Big]
      \frac{\rho_N}{\rho_{N,n} }\dd x_{n+1}^N 
    \\
    &=  - \int \Big[ [v_N]_1^n \Big( (\text{div}_1^n[v_N]_1^n) + 2 [v_N]_1^n\cdot [u_N]_1^n\Big) + 
   \\
    &\hspace{15em}  + [v_{N,n}]_1^n \Big( \text{div}_1^nv_{N,n} + 2 v_{N,n}\cdot u_{N,n}\Big)
    \Big] \frac{\rho_N}{\rho_{N,n}}
    \dd x_{n+1}^N-\\
    &\hspace{23em}
      -\int [v_N]_1^n \Big[ \frac{\text{div}_{n+1}^N(v_N \rho_N)}{\rho_{N}}\Big]\frac{\rho_N}{\rho_{N,n} }\dd x_{n+1}^N
      .
  \end{aligned}
\end{equation}
Le us also now restrict our attention to the last term on the last line. Let us also observe that
\begin{multline*}
  \int \Big\{ -(v_N\cdot\nabla)[v_N]_1^n\Big\}\frac{\rho_N}{\rho_{N,n} }\dd x_{n+1}^N
  =\int \Big\{ -(v_N\cdot\nabla)_1^n[v_N]_1^n\Big\}\frac{\rho_N}{\rho_{N,n} }\dd x_{n+1}^N    \\
  +\int \Big\{ -(v_N\cdot\nabla)_{n+1}^N[v_N]_1^n\Big\}\frac{\rho_N}{\rho_{N,n} }\dd x_{n+1}^N,
\end{multline*}
and, integrating by parts, we get that the last term on the right-hand side can be written as
\begin{equation*}   
  \int \Big\{ [v_N]_1^n\frac{(\text{div}_{n+1}^N(\rho_N v_N)}{\rho_N}\Big\}\frac{\rho_N}{\rho_{N,n} }\dd x_{n+1}^N,
\end{equation*}
which cancels the last term in the r.h.s. of decomposition  \eqref{eq:34}.
Putting everything together we obtain the formulas in the statement of the theorem.
\end{proof}
\begin{remark}
  The second equation in \eqref{eq:61} can be written more expressively
  using conditional expectations:
  \begin{equation}
    \label{eq:62}
    \begin{aligned}
      \partial_t v_{N,n}(X_1^n,t)
      = &  -\EE \Big[ ([v_N]_1^n\cdot\nabla)_1^n [v_N]_1^n
          + [v_N]_1^n \Big( ({\div}_1^n[v_N]_1^n) + 2 [v_N]_1^n\cdot [u_N]_1^n\Big) \Big| X_1^n(t)=x_1^n \Big] +
      \\
        &\hspace{5em}  + [v_{N,n}]_1^n \Big( {\div}_1^nv_{N,n} + 2 v_{N,n}\cdot u_{N,n}\Big)+\\
        &\hspace{5em} + \int_{\R^{N-n}}\tfrac12 [\nabla]_1^n \Big[ u^2_{N}+\mathrm{div}\cdot (u_N) \Big]
          \frac{\rho_N(x_1^n,x_{n+1}^N;t)}{\rho_{N,n}(x_1^n;t)}dx^{n+1}_N +\\
        &\hspace{6em}- \frac{1}{N} \sum_{1\le j<k\le n}  \int [\nabla]_1^n V(x^j-x^k) 
          \frac{\rho_N(x_1^n,x_{n+1}^N;t)}{\rho_{N,n}(x_1^n;t)} \dd x_{n+1}^N  +\\
        &\hspace{7em}- \frac{N-n}{N} \sum_{1\le j\le n }  \int [\nabla]_1^n V(x^j-z) 
          \frac{\rho_{n+1,N}(x_1^n,z;t)}{\rho_{N,n}(x_1^n;t)} \dd z
          .
    \end{aligned}
  \end{equation}
\end{remark}
\noindent A formal infinite particles limit allows to derive the infinite Madelung hierarchy.
\begin{proposition}
  The \emph{infinite Madelung hierarchy} takes the form
  \begin{align*}\label{infiniteMadelunghierarchy}
  \partial_t\rho_n(X_1^n,t)
  = &\;-\mathrm{div}_1^n (\rho_n v_{n})(X_1^n;t),
  \\
  \partial_t v_{n}(X_1^n,t)
  = & -(v_{n}\cdot\nabla)_1^n v_{n}
      + \tfrac12 [\nabla]_1^n \Big[ u^2_{n}+\mathrm{div}\cdot (u)_{n} \Big]-\\
    &\hspace{1cm}- \sum_{1\le j\le n }  \int [\nabla]_1^n V(x_j-z) 
      {\rho_{n+1}(x_1^n,z,t)} \dd z,
\end{align*}
where
\begin{align*}
  \rho_n(x_1^n,t)&=\prod_{1\leq k \leq n}\rho(x_k,t),\\
  v_n(x_1^n,t)&=(v(x_1,t),v(x_2,t),...,v(x_n,t) ), \\
  u_n(x_1^n,t)&=(u(x_1,t),u(x_2,t),...,u(x_n,t) )
                .
\end{align*}
\end{proposition}
\begin{proof}
   The result follows formally by the property of propagation of chaos, that is by the fact that 
$\rho_{N,n}$ factorizes in the infinite particle limit $N\rightarrow\infty$, i.e.  $\rho_{N,n}\rightarrow\rho^{\otimes n}$. For the proof of propagation of chaos, see Proposition \ref{cor:Bardos_convergence}.
\end{proof}

\subsection{Schr\"odinger type solutions to finite and infinite Madelung hierarchies}
In this section we first show by straightforward calculations that the limits of the $N$-particle Madelung hierarchy coincide with a sort of factorized product of solutions of the self-interacting Madelung equation, and then we point out that to the von Neumann equation there is associated a Madelung equation both in the linear and in the nonlinear case.

Although we cannot address the uniqueness problem of the solution to the infinite Madelung hierarchy, which is not trivial because of the non-linearity, we solve the question of uniqueness by introducing the concept of admissible solution of the self-interacting Madelung equation, which are the solutions coming from a Schr\"odinger equation.

\paragraph{Admissible solutions to the Madelung equations}
The following peculiar factorization result holds.
\begin{proposition}\label{prop:factorization}
  If the pair $(\rho,v)$ satisfies the self-interacting Madelung equation \eqref{eq:29}, then the sequence of the marginal pairs
  $$
  \rho^{\otimes n}(x_1^n,t)=\prod_{1\leq k \leq n}\rho(x_k,t), \quad v_n(x_1^n,t)=(v(x_1,t),v(x_2,t),...,v(x_n,t) )
  $$
  satisfies the infinite Madelung hierarchy.
\end{proposition}
\begin{proof} The proof consists in a straightforward calculation.
\end{proof}
\begin{definition}\label{quantum_versus_proba}
  We call \textit{admissible solution}  a solution $(\rho,v)$ to the self-interacting Madelung equation
  when it is
given by 
  \begin{align*}
    \rho(x;t) &\deq \rho^Q(x,x;t),\\
    v(x;t) &\deq \frac{1}{\rho(x;t)}\Big[(-\ii)\frac12(\nabla_x - \nabla_y)\rho^Q(x,y,t)\Big]_{x=y}
             ,
  \end{align*}
  where
  \begin{equation}
      \rho^Q(x,y,t) = \psi(x,t)\overline{\psi(y,t)}
      ,
  \end{equation} 
  with $\psi(x,t)$ solution 
  to the nonlinear Schr\"odinger equation \eqref{nonlin_schro_intro}. 
   More generally, a solution to any Madelung equation which is also a solution to a Schr\"odinger equation is denoted as \emph{Schr\"odinger type} Madelung solutions. 
\end{definition}
\begin{remark}
  Note that an admissible solution is of Schr\"odinger type.
  The definition of admissible solution is convenient to select a \textit{unique} solution among all possible solutions 
  to the Madelung equations.
\end{remark}
 To better clarify the relation between the results in \cite{Bardos}
and the approach presented here, we state the following straight forward results.
\begin{theorem}[Linear case]\label{theoremA} 
  For each solution $\rho^Q(x,y,t)$ to the von Neumann equation 
  \[
    i \hbar \partial_t \rho^Q_N(x,y,t)=-\frac{(\hbar)^2}{2}[\Delta_x-\Delta_y]\rho^Q_N(x,y,t)+\frac{1}{N}\sum_{1\leq j<k\leq N} [V(|x_j-x_k|)-V(|y_j-y_k|)] \rho^Q(X_N,Y_N,t),
  \]
  there is associated a solution to the Madelung equation  \eqref{eq:12}-\eqref{eq:11} .
\end{theorem}
\begin{proof}
  Note first that we have, for all $t\in [0,T]$,
  \begin{equation*}
    1= \Tr[\rho^Q_N(t)] = \int_{\R^3N} \rho_N(x,x,t) \dd x
    .
  \end{equation*}
  Differentiating 
  the left-most and right-most terms with respect to 
  $t$ we get
  \begin{align*}
    0
    &= \int \partial_t \rho^Q_N(x,x;t) \dd x\\
    &= \int [\partial_t \rho^Q_N(x,y;t)]_{x=y} \dd x\\
    &= -\frac{\ii\hbar^2}{2} \int [(\Delta_x - \Delta_y)\rho^Q_N(x,y;t)]_{x=y} \dd x\\
    &= \frac{\ii\hbar^2}{2} \int \mathrm{div} 
      ([(\nabla_x - \nabla_y)\rho^Q_N(x,y;t)]_{x=y}) \dd x\\
    &= -\frac{\hbar^2}{2} \int_{\R^3} \mathrm{div} 
      \rho_N(x;t)v_N(x;t) \dd x
      ,
  \end{align*}
  where the linear von Neumann equation in the statement of the theorem has been used,
   which gives that  the continuity equation is satisfied.
   Taking the time derivative of the second equation in Definition \ref{quantum_versus_proba}   we get 
  \[ \partial_t(v_N(X,t)\rho_N(x,t))  = \frac12
    \Big[(-\ii)(\nabla_{X} - \nabla_{y})
    \partial_t \rho^Q_N(x,y,t)
    \Big]_{X=Y}.
  \] 
  and inserting  the von Neumann equation we obtain
  \[ \partial_t(v_N(x,t)\rho_N(x,t))  = \frac{(\hbar)^2}{4}
    \Big[(\nabla_{X} - \nabla_{y})
    [\Delta_x-\Delta_y]\rho^Q_N(x,y,t)+
  \]
  \[+ \Big[(\nabla_{X} - \nabla_{y})\frac{1}{N}\sum_{1\leq j<k\leq N} [V(|x_j-x_k|)-V(|y_j-y_k|)] \rho^Q(X_N,Y_N,t).
  \]
   Let us put $\rho^Q_N(x,y,t)=\Psi_N(x,t)\cdot \overline{\Psi_N(y,t)}$ and $
    \Psi_N(x,t)=\exp({R})\cos{S}+\ii \exp({R})\sin{S}$.
  By first considering  the potential term and by expressing the quantum density matrix in terms of $R$ and $S$
  an  easy calculation gets 
  \[
    \Big[(\nabla_{x} - \nabla_{y})\frac{1}{N}\sum_{1\leq j<k\leq N} [V(|x_j-x_k|)-V(|y_j-y_k|)] \rho^Q(X_N,Y_N,t)\Big]_{X=Y}=
\frac{2}{N}\sum_{1\leq j<k\leq N} [\nabla_{x}V(|x_j-x_k|)] \rho_N(x,t)
  \]
  Focusing now on the kinetic part easy calculations provide  with $\rho_N=\exp{2R},$
  \[
    \partial_t(v_N(x,t)\rho_N(x,t)) =\frac{1}{2}\rho_N \Big[2 \nabla R \Delta R+\nabla \Delta R -4 \nabla R |\nabla S|^2 -4 \nabla S \Delta S\Big],
  \] 
  which gives
  \[
    \partial_t(\rho_N(x,t)) v_N(x,t)+ \rho_N(x,t) \partial_t(v_N(x,t))=\rho_N[\frac{1}{2}\nabla (|\nabla R|^2+ \Delta R]- \rho_N[(\nabla S)(2 \nabla R \nabla S + 2 \Delta S]-\rho_N (\nabla S \Delta S).
  \]
  Putting the two parts together we finally obtain
  \begin{equation}\label{eq:35} 
    v_N(x,t)\Big[\partial_t(\rho_N(x,t)) - \nabla \cdot (\rho_N(x,t) v_N(x,t))\Big]+
  \end{equation}
  \begin{equation*}
    \rho_N(x,t)\Big[ \partial_tv_N(x,t)+(v_N\cdot \nabla)v_N(x,t)-\frac{1}{2}\nabla (|u_N|^2+ \nabla \cdot u_N)+\nabla V\ast \rho_N(x,t)\Big]=0.
  \end{equation*}
\end{proof}
A similar result holds for the asymptotic nonlinear case.
\begin{theorem}[Non-linear case]\label{theoremB}
  For each solution $\rho^Q(x,y,t)$ to the asymptotic von Neumann equation 
  \[
    i \hbar \partial_t \rho^Q(x,y,t)=-\frac{(\hbar)^2}{2}[\Delta_x-\Delta_y]\rho^Q(x,y,t)+\int [V(|x-z|)-V(|y-z|)] \rho^Q(z,z,t)\cdot \rho^Q(x,y,t),
  \]
  there is associated a solution to the asymptotic Madelung equation \eqref{eq:29}.
\end{theorem}
\begin{proof}
The continuity equation can be derived as in Theorem \ref{theoremA}.
Taking the time derivative of the second equation in Definition \ref{quantum_versus_proba}  
and inserting  the von Neumann equation we obtain
  \[ \partial_t(v(x,t)\rho(x,t))  = \frac{(\hbar)^2}{4}
    \Big[(\nabla_{x} - \nabla_{y})
    [\Delta_x-\Delta_y]\rho^Q(x,y,t)+
  \]
  \[+\int [V(|x-z|)-V(|y-z|)] \rho^Q(z,z,t)\cdot \rho^Q(x,y,t),
    \Big]_{x=y}.
  \]
  Putting $\rho^Q(x,y,t)=\Psi(x,t)\cdot \overline{\Psi(y,t)} $,
  an  easy calculation for  the potential term,
   gives
  \[
    \int [V(|x-z|)-V(|y-z|)] \rho^Q(z,z,t)\cdot \rho^Q(x,y,t)= [V\ast \rho(x)-V\ast \rho(y)] \rho^Q(x,y,t),
  \]
  and
  \[
    \frac{(\hbar)^2}{2}
    \Big[\nabla_{X} - \nabla_{y}\Big]\Big[[V\ast \rho(x)-V\ast \rho(y)] \rho^Q(x,y,t)\Big]=2\nabla(V\ast \rho(x)\rho(x,t).
  \]
  Since the kinetic part is the same as in Theorem \ref{theoremA} we get again that
  \[
    \partial_t(\rho(x,t)) v(x,t)+ \rho(x,t) \partial_t(v(x,t))=\rho(x,t)\Big[-(v\cdot \nabla)v(x,t)+[\frac{1}{2}\nabla (|u|^2+ \nabla \cdot u)\Big]+v(x,t) \nabla \cdot (\rho(x,t) v(x,t)).
  \]
  Putting the two parts together we finally get
  \[
    v(x,t)\Big[\partial_t(\rho(x,t)) - \nabla \cdot (\rho(x,t) v(x,t))\Big]+ \rho(x,t)\Big[ \partial_t(v(x,t))+(v\cdot \nabla)v(x,t-\frac{1}{2}\nabla (|u|^2+ \nabla \cdot u)+2\nabla V\ast \rho(x,t)\Big]=0.
  \] 
\end{proof}
\begin{remark}
  We observe that when $\rho^Q(x,y,t)$ is a distributional solution to the von Neumann equation  then both the Madelung equations are satisfied in the same distributional sense. In particular, where the probability density $\rho(x,t) $ is different from zero the non-linear Madelung equation for the current velocity holds and, analogously, where  $v(x,t)$ is different from zero, the continuity equation holds. 
\end{remark}

\section{Conservative diffusions and Nelson diffusions}

\medskip
\subsection{Conservative diffusions: Carlen theorem}
In the terminology of Nelson, the \emph{conservative diffusions} 
are a class of stochastic diffusion processes
for which the time-reversed processes is again a diffusion (\cite{NelsonQF}). This is in a similar spirit  as
the time reversal invariance in quantum mechanics.
In \cite{CarlenCD} this class is characterized in terms of \textit{proper infinitesimal
  characteristics} $(\rho_t(x),v_t(x))$
consisting of a time-dependent probability density $\rho_t(x)$ and a
time-dependent vector field $v_t(x)$ defined $\rho_t(x)dxdt$-a.e,
so constructed as to  have the time-reversal symmetry, i.e. the process is a diffusion both forward and backward in time
(see, e.g., \cite{NelsonQF}).
Carlen proved the well-posedness of this class of diffusions in
 (\cite{Carlen:162610},\cite{carlen1984conservative}). See also \cite{Follmer_Wiener_space},\cite{millet1989integration}.

 We recall the definition of the Carlen class, which contains  minimal regularity conditions
 for the existence of the related diffusion process.
\begin{definition}(\cite{CarlenCD},\cite{Carlen:162610} ) \label{conservativediffusiondefinition}   
  The space of \emph{proper infinitesimal characteristics} is the set of pairs $(\rho_t(x),v_t(x))$,
  with $\rho: \RR^d \times \RR_+\rightarrow \RR_+$, $v:\RR^d\times \RR_+\rightarrow\RR^d$ functions
  satisfying the following properties:
  the \textit{continuity equation} in weak form,
  \begin{equation}\label{eq:38} 
    \int_{\R^d}f(x,T)\rho_T(x)dx-\int_{\R^d}f(x,0)\rho_0(x)dx=\int_0^T\int_{\R^d}(v_t\cdot\nabla
    f)\rho_t(x)dx,
  \end{equation}
  and the \textit{finite energy condition},
  \begin{equation}\label{eq:39}
    \int_0^T(||\nabla\sqrt{\rho_t}||^2_{L^2}+||v_t\sqrt{\rho_t}||^2_{L^2})dt
    <+\infty,
  \end{equation}
  \noindent  with $||\cdot||_{L^2}$ the $L^2(\R^d\times
  \R, dxdt)$-norm, for all $T\geq 0$ and all $f\in
  C_0^{\infty}(\R^{d+1})$.
\end{definition}
\begin{remark}
  The equation \eqref{eq:38} is a weak form of the continuity equation,
  and is what one obtains summing the forward and backward Fokker-Planck equations \cite{carlen1984conservative}.
  The finite energy condition \eqref{eq:39}
  is equivalent to requiring that the time integral of the kinetic energy remains constant during the motion.
  Furthermore we prove that it is also related  to the time integral both of a finite relative entropy condition
  and of the finite Boltzmann entropy at the initial and final time (Lemma \ref{lem:vu}).   
\end{remark}

The following is a fundamental result which guarantees that
to any pair of proper infinitesimal characteristics there is associated a well-defined diffusion process, which is the weak solution to a Brownian motion driven SDE. 
\begin{theorem}[Carlen theorem \cite{CarlenCD,Carlen:162610,carlen1984conservative}]\label{Carlentheorem} 
Let $(\Omega, \mathcal F, \mathcal F_t,X_t) $, with
$\Omega=C(\R_+,\R^d)$, be the evaluation stochastic process
$X_t(\omega)=\omega(t)$, with $\mathcal F_t=\sigma(X_s, s\leq t)$
the natural filtration.
  If $(\rho,v)$ is a proper infinitesimal characteristics  pair in the sense of Definition \ref{conservativediffusiondefinition}, putting 
  \begin{equation}\label{eq:40} 
    b:=u+v, \quad 
    u(x,t)\deq
    \begin{cases}
      \frac{\nabla\rho_t(x)}{2\rho_t(x)}, & \rho_t(x)\ne 0\\
      0, & \rho_t(x)=0,
    \end{cases}
  \end{equation}
  then there exists a unique Borel
  probability measure $\mathbb P$ on $\Omega$ such that
  \begin{enumerate}[(i)]
  \item $(\Omega, \mathcal F, \mathcal F_t,X_t,\mathbb P) $ is a Markov
    process;

  \item the image of $\mathbb P$ under $X_t$ has density
    $\rho(t,x):=\rho_t(x)$;

  \item $W(t):=X(t)-X(0)-\int_0^tb(X(s),s)ds$ is a $(\mathbb P,\mathcal F_t)$-Brownian motion, which means that the following Brownian-driven SDE, 
    \begin{equation}\label{eq:41}
      \dd X(t) = b(X(t),t)\dd t + \dd W(t)
      ,
    \end{equation}
    admits a weak solution.
      \end{enumerate}
\end{theorem}

\subsection{Nelson diffusions}\label{subsec:Nelsondiffusions}
The Carlen diffusion class or the class of
\emph{conservative diffusions} is quite general in the sense that
not every conservative diffusion comes from a Schr\"odinger equation.
As we already said, this class has been well studied from the probabilistic point of view (\cite{Follmer_Wiener_space},\cite{millet1989integration}).
Conversely, if we consider the infinitesimal characteristics
associated to a finite energy solution of a Schr\"odinger equation,
then these infinitesimal characteristics  satisfy the existence conditions.
In other words, as we shall see in Proposition  \ref{Nelsondiffusions} below,
to any good solution of a Schr\"odinger equation
it is associated a conservative diffusion. 

In this work we consider three different conservative diffusions:
one associated with the $N$-particle linear Schr\"odinger equation \eqref{schro_intro},
one associated to the one-particle nonlinear Schr\"odinger equation \eqref{nonlin_schro_intro},
and finally one that in some sense interpolates between the previous two.
The latter could be interpreted as a stochastic analog to the procedure described in \cite{Bardos}
and represents the stochastic version of the formal  Madelung hierarchies that we proposed in Section~\ref{sec:madelung}.
After introducing these three kind of diffusions, for convenience of the reader here all together, we shall prove that they are indeed well defined.

In order to introduce a conservative diffusion, by Carlen theorem \ref{Carlentheorem},
it is enough to define a pair of functions $(\rho,v)$ satisfying Definition \ref{conservativediffusiondefinition}.
First, let us introduce  conservative diffusions for the linear and non-linear Schr\"odinger equations, respectively.
Let $\Psi_N\in C([0,T],H^1(\RR^{Nd}))$ be a solution of the $N$-particle linear Schr\"odinger equation
\eqref{schro_intro} and let $\psi \in C([0,T],H^1(\RR^{d}))$ be a solution of the nonlinear Schr\"odinger equation \eqref{nonlin_schro_intro} (Theorem \ref{Schroedingerconservativesolution} and Theorem \ref{well}).
We consider here the special situation where
\begin{equation}\label{eq:42}
\begin{aligned}
   \rho_N(x^N,t)&\equiv |\Psi_N(x^N,t)|^2 ,
    \\ 
    \rho(x,t)&\equiv |\psi(x,t)|^2
    , 
\end{aligned}
\end{equation}
for $t\in[0,T]$, $x^N\in\RR^d$, $x\in\RR^{d}$.
Moreover let us introduce the current velocities given by
\begin{equation}\label{eq:43}
\begin{aligned}
  v_N &\equiv\frac{j_N}{\rho_N} = \frac{1}{2\ii\rho_N}(\overline\Psi_N\nabla\Psi_N - (\nabla\overline\Psi_N)\Psi_N),\\
  v&\equiv \frac{j}{\rho} = \frac{1}{2\ii\rho_N}(\overline\psi\nabla\psi - (\nabla\overline\psi)\psi)
     .    
\end{aligned}
\end{equation}
Note that these objects are well defined (see Proposition \ref{pro:H1} below).
As we shall establish in Proposition \ref{Nelsondiffusions}, the pairs $(\rho_N,v_N)$ and $(\rho,v)$,
defined in terms of $\Psi_N$ and $\psi$ respectively, define a conservative diffusion.
In particular, according with Theorem \ref{Carlentheorem}, to the infinitesimal characteristics the following well defined SDEs are uniquely associated:
\begin{equation}\label{eq:44} 
  \begin{aligned}
    \dd X_N(t) &= b_N(X_N(t),t)\dd t + \dd W_N(t),\quad b_N = u_N+v_N,\quad u_N=\frac{\nabla\rho_N}{2\rho_N},\\
    \dd X(t) &= b(X(t),t)\dd t + \dd W(t),\quad b = u+v,\quad u= \frac{\nabla\rho}{2\rho}
               ,
  \end{aligned}
\end{equation}
where $W_N$ is a $Nd$-dimensional Brownian motion and $W$ a $d$-dimensional one.
We shall call \textit{Nelson diffusions} the diffusions in \eqref{eq:44}.
They are a special case of conservative diffusions which are directly associated to the solution of a (linear or nonlinear)
Schr\"odinger equation.
We shall refer to the first SDE in \eqref{eq:44}
as the Nelson diffusion associated to the $N$-particle Schr\"odinger equation
and, similarly, we shall refer to the second one
as the Nelson diffusion associated to the one-particle Schr\"odinger equation.

We notice that the first SDE in \eqref{eq:44} describes a system of $N$-interacting diffusions
where the drift coefficient depends on all the $N$ process components.
Since our final goal is to send $N\to\infty$,
in order to do so, we propose a conditioning strategy allowing to consider,
so to say,
only the first $n\le N$ components of the stochastic system.    
This conditioning strategy will produce a Nelson diffusion which
\emph{interpolates} between the two Nelson diffusions defined above.
We start by noting that we can write the variables $u_{N,n}$ and $v_{N,n}$
of Definition \ref{marginal_Madelung_variables}
as conditional expectations
\begin{equation}\label{eq:45}
  \begin{aligned}
    u_{N,n}(x^n_1,t) 
    &=\mathbb E\Big[ [u_N(X^N(t),t)]_{1}^n \Big| X_1^n(t)=x_1^n\Big]
      , 
    \\[.5em]
    v_{N,n}(x_1^n,t)
    &=\mathbb E\Big[ [v_N(X^N(t),t)]_{1}^n \Big| X_1^n(t)=x_1^n\Big]
      ,
  \end{aligned}
\end{equation}
where $X_N$ is the Nelson diffusion associated to the $N$-particle Schr\"odinger equation \eqref{schro_intro}.
\begin{remark}\label{rem:rho_u_Nn}
  Considering the marginal Madelung variables given in Definition \ref{marginal_Madelung_variables},
 we note, that
  \begin{equation}
    \label{eq:46}
    u_{N,n} = \frac{\nabla \rho_{N,n}}{2\rho_{N,n}}
    ,
  \end{equation}
  which means that the relation between $u_{N,n}$ and $\rho_{N,n}$ is the same as the usual one in the definition of conservative diffusion.
  This justifies a posteriori our conditioning procedure.
\end{remark}
As we shall prove rigorously below, $(\rho_{N,n},v_{N,n})$ defines a conservative diffusion $X_{N,n}(t)$, $t\in[0,T]$,
which we shall call \textit{conditioned Nelson diffusion}.
Because of \eqref{eq:46}, this conditioned Nelson diffusion satisfies
a similar SDE as in the case of the two Nelson diffusions related to Schr\"odinger equations introduced above, given by:
\begin{equation}\label{eq:47}
  \dd X_{N,n}(t) = b_{N,n}(X_{N,n}(t),t)\dd t + \dd W_n(t),
  \quad b_{N,n} = u_{N,n}+v_{N,n}
  ,
\end{equation}
where $u_{N,n}$ and $v_{N,n}$ are as in \eqref{eq:45}
and $W_n$ is a $nd$-dimensional Brownian motion.
\begin{remark}
  Let us motivate the introduction of the conditioned Nelson diffusions with the following  argument.
  Let us assume that we start from an $Nd$-dimensional Nelson diffusion \eqref{eq:44},
  but we consider random variables only of the form  
  $f(X_1^n(t))$ where $f$ is a regular function $f:\RR^n\rightarrow\RR$.
  These observables depend only on the first $n$ components of the solution $X_N(t)$
  to the $Nd$-dimensional diffusion \eqref{eq:44}, and we
   observe that 
  \begin{equation}\label{eq:63}
    \EE[f([X_N]_1^n(t))]
    =\EE[f(X_{N,n}(t))]
    ,
  \end{equation}
  where $X_{N,n}$ is a solution of the conditioned Nelson diffusion \eqref{eq:47}.
  This can be easily explained for example by the following simple computation.
  First note that the generator $(b_N\cdot\nabla -\tfrac12\Delta)$
  of the $Nd$-Nelson diffusion satisfies the following
  \begin{align*}
    \EE[g(X_1^n(t))\partial_t f(X_1^n(t))]
    &= \EE[g(X_1^n(t)) (b_N(X^N)\cdot\nabla -\tfrac12\Delta) f(X_1^n(t))]\\
    &= \EE[\EE[g(X_1^n(t)) (b_N(X^N)\cdot\nabla -\tfrac12\Delta) f(X_1^n(t)) | X_1^n(t) ]]\\
    &= \EE[ g(X_1^n(t)) (\EE[ b_N(X^N)| X_1^n(t) ]\cdot\nabla -\tfrac12\Delta) f(X_1^n(t)) ]\\
    &= \EE[ g(X_1^n(t)) ( b_{N,n}(X_1^n)\cdot\nabla -\tfrac12\Delta) f(X_1^n(t)) ]\\
    &= \int_{\RR^{Nd}} g(x_1^n(t)) ( b_{N,n}(x_1^n)\cdot\nabla -\tfrac12\Delta) f(x_1^n(t)) \rho_N(x^N,t)\dd x^N\\
    &= \int_{\RR^{Nd}} g(x_1^n(t)) ( b_{N,n}(x_1^n)\cdot\nabla -\tfrac12\Delta) f(x_1^n(t)) \rho_{N,n}(x_1^n,t)\dd x_1^n\\
    &= \EE[ g(X_{N,n}(t)) ( b_{N,n}(X_{N,n})\cdot\nabla -\tfrac12\Delta) f(X_{N,n}(t)) ]
      .
  \end{align*}
  In the computation above we used, 
  in the first line, the definition of generator,
  in the second, the properties of conditional expectation,
  in the third, the fact that the observables depend only on $X_1^n(t)$,
  in the fourth, the definition of $b_{N,n}$. The fifth and sixth lines follow easily.
  Finally the last line is a consequence of the fact that $\rho_{N,n}(t)$ is indeed the marginal distribution at time $t$
  of the process $X_{N,n}$ which solves the conditioned Nelson diffusion \eqref{eq:47}.
  This computation justifies \eqref{eq:63} which, in turn,
  support the introduction of the conditioned Nelson diffusions.
\end{remark}

We finally prove that the three Nelson diffusion processes just introduced
are well defined.  
\begin{proposition}\label{Nelsondiffusions}
  The three pairs of infinitesimal characteristics
  $(\rho_N,v_N)$, $(\rho,v)$, $(\rho_{N,n},v_{N,n})$ defined above
  are conservative diffusions according with Definition \ref{conservativediffusiondefinition}.
\end{proposition}
\begin{proof}
  The proof follows from some regularity results which, for convenience, we collect in the following subsection.
  The fact that the infinitesimal characteristics pair $(\rho_N,v_N)$, associated to the $N$-particle Schr\"odinger equation \eqref{schro_intro}
  satisfies the conditions in  Definition \ref{conservativediffusiondefinition}, is a direct consequence of
  Theorem \ref{Schroedingerconservativesolution} below.
Let us now turn our attention to the infinitesimal characteristics pair $(\rho,v)$
  associated to the one-particle nonlinear Schr\"odinger equation \eqref{nonlin_schro_intro}.
  The continuity equation  in  Definition \ref{conservativediffusiondefinition}
  holds also in this case. Indeed note that the derivation of the continuity equation
  from the (linear or nonlinear) Schr\"odinger equation is independent on the
  potential (cf. Theorem \ref{theoremA},Theorem \ref{theoremB}). Hence the continuity equation will hold
  as soon as the integrals in \eqref{eq:38} are well defined
  in the distributional sense. The fact that these integrals are well defined
  is a consequence of the first point in Proposition \ref{pro:H1} below.
  Finally, the finite energy condition in Definition \ref{conservativediffusiondefinition},
  follows from the second point in Proposition \ref{pro:H1} below.

  Similar results hold for the infinitesimal characteristics pair $(\rho_{N,n},v_{N,n})$.
  Indeed it satisfies the continuity equation \eqref{eq:33},
  where each term is well defined again by the first point of Proposition \ref{pro:H1},
and the finite energy condition derives from the third point in Proposition \ref{pro:H1}.
  The proof is now complete.  
\end{proof}

\subsection{Some analytical results}
After having recalled two useful regularity theorems for the linear and nonlinear Schr\"odinger equations here considered, in this section we establish the well-posedness of the class of conditioned Nelson diffusions.  
\paragraph{Linear Schr\"odinger equation}
For convenience, we reproduce here a theorem due to Carlen which applies to any linear Schr\"odinger equation.
\begin{theorem}[\cite{carlen1984conservative}]\label{Schroedingerconservativesolution}
  Let  $\Psi$  be a solution to the Schr\"odinger equation \eqref{schro_intro}
  with the potential $V$ belonging to the Rellich class potential and with initial conditions $\Psi_0$
  such that $||\nabla \Psi_0 ||^2<\infty$, with  $||\cdot ||$ the norm in $ L^2(\R^d).$
  Then
  \begin{enumerate}[(i)]
  \item For all $t$, $ ||\nabla \Psi_t||^2<\infty,$ and $t \rightarrow |\nabla \Psi_t |^2$ is continuous.
  \item There are unique jointly measurable functions $\Psi(x,t)$ and $\nabla \Psi(x,t)$ such that $\Psi(x,t)=\Psi_t$ and $\nabla \Psi(x,t)=\nabla \Psi_t$ a.e. $x$ with respect to $ dx$.
  \item With 
    \begin{equation}\label{eq:48}
      u(x,t)
      =
      \begin{cases}
        \mathcal{R}e\frac{\nabla \Psi(x,t)}{\Psi(x,t)}, & \Psi(x,t)\ne 0,\\
        0, &\Psi(x,t)=0,
      \end{cases}
      \quad
      v(x,t)=
      \begin{cases}
        \mathcal{I}m\frac{\nabla \Psi(x,t)}{\Psi(x,t)}, & \Psi(x,t)\ne 0,\\
        0, &\Psi(x,t)=0
             ,
      \end{cases}
    \end{equation}
    where $\Psi(x,t)\neq 0$ and equal to zero otherwise, and $\rho(x,t)=|\Psi(x,t) |^2,$ for each finite interval $[0,T]$ there exists a constant $M$ such that for a.e. $t\in [0,T].$ 
    \[
      \int(u^2+v^2)\rho(x,t)dx < M. 
    \]
    Furthermore, the continuity equation holds:
    for all real bounded functions $f$ on $\R^d$ with bounded first derivatives, $t \rightarrow \int f(x)\rho(x,t)dx$ is continuously differentiable and
    \[
      \frac{d}{dt} \int f(x)\rho(x,t)dx=\int (v(x,t)\cdot \nabla f(x))\rho(x,t)dx
    \]
  \item 
    If $\int{x^2\rho(x,0)dx}$ is finite, then the density $\rho(x,t)$ has finite second moments, uniformly bounded in any compact interval.
  \end{enumerate}
\end{theorem}

Let us call the Nelson map the map that associates a conservative diffusion to each good solution of the Schr\"odinger equation, as stated in the previous theorem.

\begin{remark}
  The Rellich class potential is a general class of potentials in quantum mechanics. In $\R^3$ it is  equivalent to assume that  $V \in L^2(\R^3)+L^\infty(\R^3)$).
\end{remark}
\begin{remark}
  We observe that in \eqref{eq:48} $u$ and $v$ are defined pointwise.
  To be absolutely rigorous, we observe that, in quantum mechanics,
  the wavefunction $\Psi$ (for a fixed time $t$),
  is defined only
  as an equivalence class $\Psi\in H^1(\RR^d)$ (where we denote the equivalence class still by $\Psi$).
  Therefore at first glance \eqref{eq:48} could look problematic.
  Nevertheless, since two $\Psi$ in the same equivalence class give rise to $u,v$
  in the same equivalence class, we can use \eqref{eq:48}
  to define $u,v$ also in the case where $\Psi\in H^1(\RR^d)$ is only defined up to equivalence.
  We shall rigorously specify the spaces where $u,v$ live in Proposition~\ref{pro:H1}.
\end{remark}

\paragraph{Nonlinear Schr\"odinger equation}
Even thought Nelson stochastic mechanics has been initially introduced for linear Schr\"odinger equations,
the rigorous result on the well-posedness of conservative diffusions due to Carlen
(Theorem \ref{Carlentheorem}) holds also for nonlinear Schr\"odinger equations,
as we have seen in the subsection above.
Here we start by quoting 
a theorem regarding the regularity of a solution to the nonlinear Schr\"odinger equation
(cf. \cite[Proposition 3.1, Theorem 3.1]{GinVel} and also \cite[Theorem 1.1]{Bardos}).
\begin{theorem}
  \label{well}
  Suppose the interaction potential to be rotationally invariant $V(x)=V(|x|)$, real-valued, and of the form
  \begin{equation*}
    V(|x|) = V_1 (|x|) + V_2 (|x|) \ \text{with} \  V_1 (|x|) \in L^2 (\mathbb{R}^3) \ \text{and} \ V_2 (|x|) \in L^{+\infty} (\mathbb{R}^3)
    .
    \end{equation*}
    Then, for any initial condition  $\psi (\cdot, 0) = \psi^I$, with $\psi^I \in H^1 (\R^3)$,
  the  one-particle nonlinear Schr\"odinger equation \eqref{nonlin_schro_intro}
  has a unique solution $\psi \in C([0,T]; H^1(\R^3))$.

  Moreover, the solution satisfies
  \begin{enumerate}
  \item  the unitarity condition: $\|\psi(t)\|_{L^2(\RR^{3})}= \|\psi(0)\|_{L^2(\RR^{3})}$, $t\in\RR_+$;
  \item  the total energy conservation:
  \begin{align*}
    \mathcal E(\psi(t)) = \mathcal E(\psi(0)),
  \end{align*}
  where
  \begin{align*}
    \mathcal E(f) \deq \|\nabla f\|^2_{L^2(\RR^3;\RR^3)} + P(f),
    \quad P(f) \deq \tfrac12\int |f(x)|^2 V(x-y)|f(y)|^2\dd x \dd y,\quad f\in H^1(\RR^3).
  \end{align*}
  \end{enumerate}
\end{theorem}

\paragraph{Conditioned Nelson diffusions}
The following proposition,  collecting various regularity results
which we need, shows in particular that the conditioned Nelson diffusions are well defined.
\begin{proposition}\label{pro:H1}  
  Let us denote by $\psi$ either the solution of the  Schr\"odinger equation \eqref{schro_intro},
  where for brevity fix $N=1$ (and $d$ arbitrary), or
  the solution to the non-linear Schr\"odinger equation 
  \eqref{nonlin_schro_intro}, either with initial condition
  $\psi^I\in H^1(\R^{d})$.
  
  Then the following statements hold.
  \begin{enumerate}   
  \item  
    The probability density $\rho \deq |\psi|^2\in L^1(\R^{d})$ and the current density
    $j \deq \frac{1}{2\ii}(\overline\Psi_N\nabla\psi - \psi\nabla\overline\psi)
    \in L^1(\R^{d};\R^{d})$.
  \item The osmotic velocity
    $u \deq \frac{\nabla\rho}{2\rho}\in L^2(\R^d,\rho\dd x;\R^{d})$
    and the current velocity $v \deq \frac{j}{\rho}\in L^2(\R^{d},\rho\dd x;\R^{d})$.
  \item Let us now denote by $\Psi_N$, $N\in\NN$, the solution to the 
  $N$-particle Schr\"odinger equation \eqref{schro_intro} with initial condition
  $\Psi^I_N\in H^1(\R^{Nd})$ 
  and consider the related marginal quantities,
    as introduced in Definition \ref{marginal_Madelung_variables}, $n<N$.
    Then we have $\rho_{N,n}\in L^1(\RR^{nd})$
    and  $u_{N,n},v_{N,n}\in L^2(\RR^{nd},\rho_{N,n}(x_1^n)\dd x_1^n;\RR^{dn})$. 
    Moreover
    \begin{equation}\label{eq:49}
      \|u_{N,n}\|_{L^2(\RR^{nd},\rho_{N,n}\dd x)} \le  \|\nabla\Psi_N\|_{L^2(\RR^{Nd};\RR^{Nd})},
      \quad
      \|v_{N,n}\|_{L^2(\RR^{nd},\rho_{N,n}\dd x)} \le  \|\nabla\Psi_N\|_{L^2(\RR^{Nd};\RR^{Nd})}.
    \end{equation}
\end{enumerate}
\end{proposition}
\begin{proof}
  Let us first prove the statement $1$. The fact that $\rho\in L^1(\R^d)$ is a direct consequence of the hypothesis that $\psi\in H^1(\R^d)$. Moreover,
  by Cauchy-Schwartz inequality one sees at once that from $\psi\in H^1(\R^d)$ we get $j\in L^1(\R^d;\R^d)$. 
  This concludes the proof of $1.$

  Let us now turn to the proof of point $2.$
  For $\psi\in H^1(\R^d)$, we also have $|\psi|\in H^1(\R^d)$. Furthermore
   $\partial_j |\psi| (x) = \frac{\psi(x)}{|\psi(x)|} \partial_j\psi(x)$,
  with $\frac{\psi(x)}{|\psi(x)|}\in L^\infty(\R^d)$.
  Hence we have $\rho^{1/2}=|\psi|\in H^1(\R^d)= W^{1,2}(\R^d)$.
  In particular $\nabla(\rho^{1/2})\in L^2(\R^d;\R^d)$.
  Since 
$ \nabla(\rho^{1/2})(x) = \frac{1}{2}\rho^{-1/2}(x)\nabla\rho(x)$,
  this implies that $\rho^{-1/2}\nabla\rho\in L^2(\R^d;\R^d)$.
  Therefore 
  $\frac{\nabla\rho}{2\rho}=\frac{\rho^{-1/2}\nabla\rho}{2\rho^{1/2}}$ is a well defined element
  of $L^2(\R,\rho(x)\dd x)$.
  
Similarly, we note that 
    $|j(x)|^2 \le |\psi(x)|^2|\nabla\psi|^2$.
  Since by hypothesis $\nabla\psi\in L^2(\R^d)$, then
    $\frac{|j|^2}{\rho}
= \frac{{j}^2}{|\psi|^2}
    \le |\nabla\psi|^2$
  is a well defined element of $L^1(\R^d)$ which implies that $\frac{\mathbf j}{\rho^{1/2}}$
  is a well defined element of $L^2(\R^d)$.
  Therefore, as before, we obtain that
  $\frac{\mathbf j}{\rho}$ is a well defined element in $L^2(\R^d,\rho(x)\dd x)$.

  We now turn to the statement $3.$
  First note that $\rho_{N,n}\in L^1(\RR^{nd})$ is a direct consequence of the fact
  that $\rho_N\in L^1(\RR^{Nd})$ from point 1.
  Now to prove that $u_{N,n},v_{N,n}\in L^2(\RR^{nd},\rho_{N,n};\RR^{nd})$
  it is enough to establish the bounds in \eqref{eq:49}.

  For $u_{N,n}$ we have:
  \begin{align*}
    |u_{N,n}|^2 \rho_{N,n}
    & \le \frac{1}{\rho_{N,n}^2} \Big(\int u_N \rho_N \dd x_{n+1}^N \Big)^2 \rho_{N,n}
     = \frac{1}{\rho_{N,n}} \Big(\int u_N \rho_N \dd x_{n+1}^N \Big)^2   \\
    & = \frac{1}{\rho_{N,n}} \Big(\int \frac{\nabla\sqrt{\rho_N}}{\sqrt{\rho_N}} \rho_N
      \dd x_{n+1}^N \Big)^2  
     = \frac{1}{\rho_{N,n}} \Big(\int (\nabla\sqrt{\rho_N}) (\sqrt{\rho_N})
      \dd x_{n+1}^N \Big)^2   \\
    & \le \frac{1}{\rho_{N,n}} \Big(\int (\nabla\sqrt{\rho_N}  )^2 
    \dd x_{n+1}^N\Big)
      \Big(\int \rho_N  \dd x_{n+1}^N \Big) 
     =  \Big(\int (\nabla\sqrt{\rho_N}  )^2 \dd x_{n+1}^N\Big)
      .
  \end{align*}
  Finally for $u_{N,n}$ we get,
  \begin{align*}
    \|u_{N,n}\|^2_{L^2(\RR^{nd},\rho_{N,n};\RR^{nd})}
    &= \int |u_{N,n}|^2 \dd x_1^n
    \le \|\nabla \Psi_N\|^2_{L^2(\RR^{Nd};\RR^{Nd})},
  \end{align*}
  where in the last inequality we used the computation above and the fact that
  $\|\nabla\sqrt{\rho_N}\|_{L^2(\RR^{Nd};\RR^{Nd})} = \|\nabla|\Psi_N|\|_{L^2(\RR^{Nd};\RR^{Nd})}
  = \|\nabla\Psi_N\|_{L^2(\RR^{Nd};\RR^{Nd})}$.

  \noindent Analogously, for $v_{N,n}$, we can write
  \begin{align*}
    |v_{N,n}|^2\rho_{N,n}
    &= \frac{1}{\rho_{N,n}} |j_{N,n}|^2
    \le \frac{1}{\rho_{N,n}} \Big( \int |j_N| \dd x_{n+1}^N \Big)^2\\
    &\le \frac{1}{\rho_{N,n}} \Big( \int |\nabla\Psi_N|\sqrt{\rho_N} \dd x_{n+1}^N \Big)^2\\
    &\le \frac{1}{\rho_{N,n}} \Big( \int |\nabla\Psi_N|^2 \dd x_{n+1}^N \Big)
      \Big(\int\rho_N \dd x_{n+1}^N \Big)
    =  \int |\nabla\Psi_N|^2 \dd x_{n+1}^N
      ,
  \end{align*}
  that is
  \begin{align*}
    \|v_{N,n}\|^2_{L^2(\RR^{nd},\rho_{N,n};\RR^{nd})}
    &\le \|\nabla \Psi_N\|^2_{L^2(\RR^{Nd};\RR^{Nd})}
      ,
  \end{align*}
  which concludes the proof.
\end{proof}

\section{Convergence results}

In this section we provide some rigorous convergence results.
The first concern the convergence of the marginal infinitesimal characteristics for every fixed $t$,
more precisely
the convergence of the fixed-time $n-$marginal of the joint distribution of the $N$ Bose particles, in the sense of weak convergence of probability measures,
and  of the current densities, in the sense of the weak$\ast$ convergence,
to the corresponding tensorized objects.

A further result states the weak convergence on the path-space
of the overall process probability law to the limit process law associated to the nonlinear Schr\"odinger equation via Nelson map.

\subsection{General convergence in law results for diffusion processes}
 In order to achieve the final goal of proving the convergence of the probability law of our $n$-particle interacting diffusion system to the limit law on the path space, in this section we need to discuss some preliminary convergence results which can be stated by collecting some recent achievements available in the literature and in establishing a useful decomposition of the relative entropy in terms of the kinetic energy and of the Boltzmann entropy.

Let us first reproduce here a fact regarding relative entropy
of a Brownian motion driven SDE, which we will use repeatedly in the following.

\begin{proposition}\label{proposition_follmer}
    Let $\dd Y(t) = c(Y,t)\dd t + \dd W(t)$ be a Brownian SDE, that is, an SDE driven by the Brownian motion $W(t)$, $t\in[0,T]$,
    which admits a unique square integrable weak solution $Y$,
    $Y\in L^2(\Omega,\WW;\RR^d)$, defined on the path space $\Omega=C([0,T],\RR^d),$ where $\WW$ 
    denotes the law of the Brownian motion $W$.
    Furthermore, assume that the law $\QQ$ of the solution $Y$ is absolutely continuous with respect to $\WW$.
    Then
    \begin{align*}
        \mathcal H(\QQ|\WW) = \EE[\int_0^T |c(Y,t)|^2 \dd t]
        .
    \end{align*}
\end{proposition}
\begin{proof}
   For a proof we refer to \cite{Follmer_Wiener_space}
(cf.\ also \cite{follmer1988random}, \cite{Cattiaux2021TimeRO}). 
\end{proof}

We also recall a general result that gives a useful form of weak convergence of stochastic integrals.
\begin{proposition}\label{pro:drifts}
  Consider a sequence of diffusion laws $\PP_M$, $M\in\NN$, on the same filtered space $(\Omega,\mathcal F_{[0,T]})$, where each $\PP_M$ is absolutely continuous with respect to the Wiener measure $\WW$ on $\Omega$,
  weakly converging to a probability measure $\tilde \PP$ with $\sup_{M\in\NN}\KL(\PP_M|\mathbb W)<+\infty$.
  Then, for any bounded continuous function $K:[0,T]\times\RR^{}\rightarrow\RR^{d} $ we have
    \begin{align*}
  \lim_{M\rightarrow\infty}   \EE_{\PP_M}\Big[ \int_0^T \langle \dd X(t),  K(X(t),t)\rangle \dd t \Big]
  =
    \EE_{\tilde \PP}\Big[ \int_0^T \langle \dd  X(t),  K( X(t),t)\rangle \dd t \Big]
    .
  \end{align*}
  where $X(t), t\in [0,T]$ denotes the canonical process on $\Omega.$
\end{proposition}
\begin{proof}
  See (\cite{DeVecchi_Rigoni2024}, Lemma 2.24).
\end{proof}

In the specific probabilistic framework considered here, the projection properties of Hilbert spaces translates into the minimization property of Markovian projections.

\begin{proposition}\label{pro:time_proj}
  Let $\PP$, a probability measure on the probability space $\Omega=C([0,T];\RR^d),$
  be the law of an $\RR^d$-valued stochastic process $X$ with continuous paths
  and let $\mu_t=\Law(X_t)$, $t\in[0,T]$ be the marginal law at time $t$ of the process $X$,
  i.e. the push forward of $\PP$ under the measurable map which sends each path $\omega\in\Omega$
  to its value $\omega(t)$ at time $t$.
  Consider a measurable function $\mathcal D:\Omega \times [0,T]\rightarrow\RR$
  such that $\EE[ \int_0^T |\mathcal D(X_t,t)|^2 \dd t ] <+\infty$.
  Then there exists a measurable function $D:\RR^d\times [0,T]\rightarrow\RR$
  that satisfies the following properties.  
  \begin{enumerate}
  \item $\int_0^T\int_{\RR^d} |D(x,t)|^2\mu_t(\dd x)\dd t < +\infty$;
  \item For any $f\in C_b([0,t]\times\RR^d)$ it satisfies
    \begin{equation*}
      \int_0^T \int_{\RR^d} D(x,t)f(x,t) \mu_t(\dd x)
      = \EE\Big[ \int_0^T \mathcal D(\cdot,t) f(X_t(\cdot),t)\dd t \Big]
    \end{equation*}

    In particular we have
    \begin{equation*}
      \int_0^T \int_{\RR^d} |D(x,t)|^2 \mu_t(\dd x)
      \leq \EE_\PP[ \int_0^T |\mathcal D(\cdot,t)|^2 \dd t ]
    \end{equation*}
    where the equality holds if and only if $\mathcal D(\cdot,t) = D(X_t(\cdot),t)$.
  \end{enumerate}  
\end{proposition}
\begin{proof}
  Cf. (\cite{DeVecchi_Rigoni2024}, Lemma 2.15).
\end{proof}

The following theorem is a consequence of the results in \cite{DeVecchi_Rigoni2024}.
For convenience of the reader we state the theorem in a way which is useful to our goal
and we provide a proof employing the general propositions recalled above.
The probabilistic framework is that of a sequence of laws associated with Brownian SDEs.
The theorem states that, given a family of diffusions, if we are able to control the relative entropy,
if the time marginal probability densities converge, and if the drift converges in a very weak sense,
then the diffusion family converges in law. 
\begin{theorem}\label{Th:DvR}
  Let $Y_N$, $N\in\NN$, denote a family of diffusions,
  which are solutions to $n$-dimensional SDEs driven by Brownian motion,  with drifts $c_N$ and law $\PP_N$.
  Moreover, let $Y$ be another $n$-dimensional diffusion,
  solution to a Brownian SDE with drift $c$ and law $\PP$.
  We assume the following conditions hold: 
  \begin{enumerate}[(i)]
  \item $\mathcal H(\PP_{N}| \mathbb W)<+\infty$, for all $N\in\NN$
    and $\mathcal H(\PP| \mathbb W)<+\infty$ and
    \begin{align*}
      \limsup_{N\in\NN} \mathcal H(\PP_{N}| \mathbb W) \le \mathcal H(\PP_{\infty}| \mathbb W)
      .
    \end{align*}
    \item The marginals $\mu^N_t$ of $Y_N(t)$, $N\in\NN$, and $\mu_t$ of $Y(t)$
      exist $\forall t \in [0,T]$ and satisfy
      $\lim_{N\rightarrow\infty} \mu^N_t = \mu_t$, $\forall t \in [0,T]$,
      weakly for $t$ Lebesgue almost surely.
  \item For every $K: \RR^n\times [0,T]\rightarrow \RR$ continuous bounded function,
  the expectations
    $\EE[\int_0^T K(Y_N(t),t) dY_N(t)]$ and
    $\EE[\int_0^\infty K(Y_\infty(t),t) dY_\infty(t)]$ 
    are well defined and we have that as $N\rightarrow\infty$
    \begin{align*}
      \EE[\int_0^T K(Y_N(t),t) dY_N(t)] \rightarrow  \EE[\int_0^T K(Y(t),t) dY(t)].
    \end{align*}
    Then $\PP_N$ converges to $\PP$ in law.
  \end{enumerate}
\end{theorem}
\begin{proof}
For readability we break the proof into several steps.
  \begin{enumerate}
  \item By $(i)$ and the property of compactness of the level sets of the relative entropy with respect to
    the weak topology of probability measures,
    there exists a subsequence $\PP_{k(N)}$ and a  probability measure
    $\tilde\PP$ such that, for $N\rightarrow\infty$, $\PP_{k(N)}\rightarrow\tilde\PP$ in law.

  \item By lower-semicontinuity property of the relative entropy we have
    \begin{equation}
    \label{eq:53}
    \KL(\tilde\PP | \mathbb W)\le \KL(\PP | \mathbb W)<+\infty
    .
  \end{equation}
  In particular we have that $\KL(\tilde\PP | \mathbb W)$ is finite.
 \item 
  This implies, see \cite{Follmer_Wiener_space}, that there exists a drift $\tilde c$ for
  a diffusion $\tilde Y$ with law $\tilde \PP$ such that
  $$\EE_{\tilde\PP }[\int_0^T |\tilde c(\tilde Y,t)|^2\dd t ]<+\infty,$$
  where $\tilde c$ is not necessarily a Markovian function of the process $\tilde Y$,
  that is, it can depend on the whole history of $\tilde Y$ and not just on $\tilde Y(t)$ at a certain time $t$.

  \item 
     By hypothesis $(ii)$ for any $F\in C_b(\RR^n\times[0,T])$  as $ N\uparrow\infty$:
  \begin{align*}
    \EE[ \int_0^T F(Y_N(t),t)\dd t]
      &\rightarrow  \EE[\int_0^T F(Y(t),t)\dd t].
  \end{align*}
  In particular, by restricting the left-hand side
  to the subsequence of measures $\PP_{k(N)}$ the limit of which is $\tilde\PP$,
  we get, by  weak convergence of probability measures,
  \begin{align*}
    \EE[ \int_0^T F(\tilde Y(t),t)\dd t]=  \EE[\int_0^T F(Y(t),t)\dd t]
    ,
  \end{align*}
  which means that $\tilde\mu_t=\mu_t$ by definition of time marginals.

  \item Because of points 1, 2, and 3, we can apply Proposition~\ref{pro:drifts}, obtaining that
  \begin{align*}
  \lim_{M\rightarrow\infty}   \EE_{\PP_{k(M)}}\Big[ \int_0^T \langle \dd Y(t),  K(Y(t),t)\rangle \dd t \Big]
  =
    \EE_{\tilde \PP}\Big[ \int_0^T \langle \dd \tilde Y(t),  K(\tilde Y(t),t)\rangle \dd t \Big]
    .
  \end{align*}
 Since  by the hypothesis $(iii)$ we have that
 \begin{align*}
  \lim_{M\rightarrow\infty}   \EE\Big[ \int_0^T \langle \dd Y(t),  K(Y_{k(N)}(t),t)\rangle \dd t \Big]
  =
    \EE\Big[ \int_0^T \langle \dd \tilde Y(t),  K(\tilde Y(t),t)\rangle \dd t \Big]
    ,
  \end{align*}
  then, in particular,
  \begin{align*}
    \EE\Big[ \int_0^T \langle \tilde c(\tilde Y,t), K(\tilde Y(t),t)\rangle \dd t \Big]
    =  \EE [ \int_0^T\langle c(Y(t),t), K(Y(t),t)\rangle  \dd t
    ,
  \end{align*}
and, by writing the right-hand side in terms of the time-marginals $\mu_t$, we get 
  \begin{equation}\label{eq:54}
    \EE\Big[ \int_0^T \langle \tilde c(\tilde Y,t), K(\tilde Y(t),t)\rangle \dd t \Big]
    = \int_0^T \int_{\RR^{d}}
    \langle c(y,t), K(y,t)\rangle  \mu_t(y) \dd y \dd t
    .
  \end{equation}

\item We apply Proposition~\ref{pro:time_proj} to the left-hand side of \eqref{eq:54} obtaining
  that there exists a function $\gamma$ such that
  \begin{align*}
    \EE\Big[ \int_0^T \langle \tilde c(\tilde X,t), K(\tilde X_t,t)\rangle \dd t \Big]
    = \int_0^T \int_{\RR^{d}}
    \langle \gamma(y,t), K(y,t)\rangle  \tilde\mu_t(\dd y) \dd t
    .
  \end{align*}
  By point $4.$ above we can replace $\tilde\mu_t$ on the right-hand side of the expression above
  with $\mu_t$ hence,
  \begin{align*}
    \int_0^T \int_{\RR^{d}}
    \langle \gamma(y,t), K(y,t)\rangle  \tilde\mu_t(\dd y) \dd t
    =
    \int_0^T \int_{\RR^{d}}
    \langle \gamma(y,t), K(y,t)\rangle  \mu_t(\dd y) \dd t
    =
    \int_0^T \int_{\RR^{d}}
    \langle c(y,t), K(y,t)\rangle  \mu_t(\dd y) \dd t
  \end{align*}
  where the last equality follows from \eqref{eq:54}.
  From this we get $\gamma = c$, $\mu_t(\dd y)\dd t$-almost everywhere. 
  Hence we can apply the inequality in Proposition~\ref{pro:time_proj},
  with $c$ and $\mu_t$ in place of $\gamma$ and $\tilde\mu_t$, obtaining
  \begin{align*}
    \int_0^T \int_{\RR^n} |c(x,t)|^2 \mu_t(\dd x)
    = \int_0^T \int_{\RR^n} |\gamma(x,t)|^2 \mu_t(\dd x)
    \leq \EE_\PP[ \int_0^T |\tilde c(\cdot,t)|^2 \dd t ] 
    .
  \end{align*}
  Now the left-hand side is equal to $\EE[\int_0^T |c(Y(t),t)|^2 ]=\KL(\PP|\WW)$
  whereas the right-hand side is equal to $\KL(\tilde\PP|\WW)$.
  Hence we obtain $\KL(\tilde\PP|\WW) \ge \KL(\PP|\WW)$.
  Now, together with point $2.$ we get
  $\KL(\tilde\PP|\WW) = \KL(\PP^\infty|\WW)$,
  that is
  \begin{align*}
    \EE[ \int_0^T |\tilde c(\tilde Y,t)|^2 \dd t ]
    = \EE[ \int_0^T |c(Y_t,t)|^2 \dd t ]
    .
  \end{align*}
  Hence $\tilde c$ and $c$ have the same $L^2(\tilde \dd\mu_t\dd t)$ norm
  and, by Proposition~\ref{pro:time_proj}, they coincide $\dd\mu_t\dd t$-almost surely.

\item By \cite{Follmer_Wiener_space}, an SDE with drift $\tilde c\equiv c$ in $L^2(\dd\mu_t\dd t)$,
  i.e. with finite relative entropy $\KL(\tilde\PP|\WW) =\KL(\PP|\WW)$,
  has a unique weak solution.  
  Therefore we obtain that $\tilde Y=Y$ in law.
  Since we have established that any convergence subsequence in point $1.$ converges to the same limit $Y$,
  this concludes the proof.
  \qedhere
\end{enumerate}
\end{proof}
We recall the Boltzmann entropy definition for the case of measures admitting densities with respect to the Lebesgue measure. See \cite{HauMis} for more details.
\begin{definition}[Boltzmann entropy]\label{def_boltzmann}
  Given a measure $\mu$ on $\RR^d$ with density $\rho$ with respect to the Lebesgue measure on $\RR^d$,
  the Boltzmann entropy $H(\mu)\equiv H(\rho)$ is defined, when it exists, as
  \begin{align*}
    H(\rho)\deq \int_{\RR^k} \rho(x)\log(\rho(x))\dd x
    .
  \end{align*}
\end{definition}
 The following lemma expresses the relative entropy between the law of a Nelson diffusion and the Wiener measure in terms of the kinetic energy and of the Boltzmann entropy evaluated at the initial and final time. 
\begin{lemma}\label{lem:vu}  
Let $\PP$ denote the law of a conservative diffusion on $\R^d$ which is the weak solution to a Brownian SDE with drift $b=u+v$  and marginal density $\rho_t$ and let $\WW$ denotes the Wiener law. Then we have that
  \begin{align*}\label{relative_entropy}
    \KL(\PP|\WW) &= \int_0^T \int |b(x,t)|^2 \dd x \dd t
                = \int_0^T \int_{\RR^d} \Big(u(x,t)^2 + v(x,t)^2\Big)\dd x\dd t + H(\rho(0)) - H(\rho(T)).
  \end{align*}
  In particular
  \begin{align*}
    \int_0^T \int v_t \cdot u_t\rho_t\dd x\dd t
    &=   H(\rho(0)) - H(\rho(T)) 
      ,
  \end{align*}
  where $\cdot$ denotes the scalar product of vectors in $\R^d$ and
  $H$ denotes the Boltzmann entropy as given in Definition~\ref{def_boltzmann}.
\end{lemma}
\begin{proof}
By Proposition \ref{proposition_follmer} we know that
$\mathcal H(\PP|\WW) = \int_0^T\int |b(x,t)|^2\dd x \dd t$.
On the other hand, since we assumed $b=u+v$ this means that
  $|b|^2 =  |u|^2 + |v|^2 + 2u \cdot v$
  where $u\cdot v$ denotes the scalar product in $\RR^d$.

Now note that by definition $u=\tfrac12\nabla \log\rho$ and, therefore, integrating by parts
and using the continuity equation we have
\begin{align*}
  \int_0^T \int v_t \cdot u_t\rho_t\dd x\dd t
  &= -\int_0^T\int \log( \rho_t) \partial_t \rho_t  dx dt\\
  &=  - \int ( \log( \rho_T)  \rho_T) + \int ( \log( \rho_0)  \rho_0)\\
  &= \mathcal H(\rho(0)) - \mathcal H(\rho(T))
  ,
\end{align*}
where in the last line the definition of Boltzmann entropy 
as given in Definition~\ref{def_boltzmann} has been used.
\end{proof}

We shall need also the following property for the Boltzmann entropy 
(for details we refer to \cite{HauMis}). See also \cite[Lemma 5.1]{DeVecchiUgolini} 
for a result in a similar spirit).
\begin{proposition}\label{pro:entropy-limsup}
    Assume the initial distributions $\rho_{N}(0)$ to satisfy for all $N$:
    \begin{align*}
      \int_{\RR^{d}} |x|^2 \rho_{N,1}(0,x) \dd x \le C_1,
    \end{align*}
    for a given constants $C_1<+\infty$. 
    Moreover, assume that $\rho_{N,1}(t)\rightarrow\rho_\infty(t)$,
    in law, $t$ almost surely.
    Then 
    \begin{align*}
      H(\rho_\infty) \le \liminf_N \frac{1}{N} H(\rho_N) 
      .
    \end{align*}
\end{proposition}
\begin{proof}
By Theorem~\ref{Schroedingerconservativesolution} point $(iv)$,
we have that the second moment of $\rho_N(t)$ is finite for all $t\in[0,T]$.
This fact and the hypothesis that  $\rho_{N,n}(t)\rightarrow\rho_\infty(t)^{\otimes n}$,
allows us to use a result in \cite[Theorem 3.4]{HauMis} 
which guaranties that 
$ H(\rho_\infty) \le \liminf_N  \frac{1}{N}H(\rho_{N}) $.    
\end{proof}

\subsection{Marginal density and current density convergence results}
Let us first recall a lemma, proved in (\cite[Lemma 2.3]{Bardos}), which will be useful in the following. Let us denote by $ C_c(\R^d)$ the set of continuous functions with compact support.
\begin{lemma}\label{lemma_bardos}
  Consider a family $l_N$, $N\in\N$, 
  of measurable functions together with a family $\mathtt F_N$
  of trace-class operators on $L^2(\RR^d)$ with kernel $F_N$, such that the following conditions hold:
  \begin{enumerate}
  \item for any compact operator with $
  \mathtt K:L^2(\R^d)\rightarrow L^2(\R^d)$,
  we have, as $N\rightarrow\infty$
  \begin{align*}
      \Tr( F_N\mathtt K) \rightarrow 0
      ,
  \end{align*} 
  where $\Tr(\cdot)$ denotes the trace on $L^2(\R^d)$;
  \item for $|h|_{\R^d}\rightarrow 0$,
    $\int_{\R^d} |F_N(x+h,x) - l_N(x)| \dd x \rightarrow 0$ uniformly in $N$.  
  \end{enumerate}   
  Then, for any $\varphi\in C_c(\R^d)$, we have
  $\int_{\R^d} l_N(x) \varphi(x) \dd x \rightarrow 0$ for $N\rightarrow\infty$.
  \qed
\end{lemma}
In the following, we prove the weak convergence of the fixed-time marginal densities $\rho_{N,n}\rightarrow\rho_n$
and the weak$\ast$ convergence of the fixed-time marginal current densities $j_{N,n}\rightarrow j_n$. The main idea is to approximate,  for each fixed $t$, $\rho_{N,n}$
and $j_{N,n}$ by functionals defined in terms of the integral kernel $\rho^Q_N(x,y,t)$
and then to take advantage of Proposition \ref{lemma_bardos} to obtain the desired convergence result.
\begin{proposition}\label{cor:Bardos_convergence}
  Let $\psi_N$ be the solution of the Schr\"odinger equation \eqref{schro_intro} and let $\psi$ be the solution to the nonlinear Schr\"odinger equation \eqref{nonlin_schro_intro} with initial conditions $\Psi^I_N,\psi^I$, respectively, such that: 
  \begin{itemize}
      \item[(i)] $\Psi_N^I=(\psi^I)^{\otimes N}\in H^1(\R^{Nd})$, $\psi^I\in H^1(\R^d)$,
      \item[(ii)] as $N\uparrow \infty$, $\frac{1}{N}\int |\nabla \Psi_N^I|^2 \dd x \longrightarrow  \int |\nabla \psi^I|^2 \dd x$.
  \end{itemize}
  
  Let $\rho_{N,n}$ and $v_{N,n}$ be the marginal densities and the marginal current velocities defined as in Definition~\ref{marginal_Madelung_variables} and 
  let $\rho$, $v$ be defined in terms of the solution $\psi$ of the nonlinear Schr\"odinger equation as in \eqref{eq:42} and \eqref{eq:43}.
  Then, for every fixed $t\in\R_+$
  and for any $\varphi\in C_c(\RR^{nd})$, 
  $\phi\in C_c(\RR^{nd};\RR^{nd})$, we have as $N \uparrow \infty$, 
  \begin{equation}
    \label{eq:50}
    \int_{\RR^{nd}} \rho_{N,n}(x^n,t) \, \varphi(x^n) \, \dd x^n
    \rightarrow \int_{\RR^{nd}} \rho^{\otimes n}(x^n,t) \, \varphi(x)
    \,\dd x^n 
   , 
  \end{equation}
  and
  \begin{equation}
    \label{eq:51}
    \int_{\RR^{nd}} j_{N,n}(x^n,t)\cdot\phi(x^n)\,\dd x^n
    \rightarrow \int_{\RR^{nd}} j_{\infty,n}(x^n,t)\cdot \phi(x^n)\,\dd x^n 
    ,
  \end{equation}
  where
  $j_{\infty,n} \deq \frac{1}{2\ii}(\overline\psi^{\otimes n}\nabla_n\psi^{\otimes n} - (\nabla_n\overline\psi^{\otimes n})\psi^{\otimes n})$,
  with $\nabla_n$ the Laplacian on $\R^{nd}$.
\end{proposition}
\begin{proof}
  For simplicity we remove the dependence on the fixed time $t$.
To show \eqref{eq:50} it is enough to apply 
Lemma~\ref{lemma_bardos} with $l_N=\rho_{N,n}$ and $F_N(x^N,y^N)=\rho_{N,n}^Q(x_1^n,y_1^n)$, 
where we use the notation 
   $\rho^Q_{N,n}(x^n,y^n) = \int_{\RR^{(N-n)d}} \overline{\Psi_N(x^n, x_{n+1}^N)}\Psi_N(y^n,x_{n+1}^N) dx_{n+1}^N$.
   Assumption $1.$ in the Lemma follows from
   \cite[Theorem 4.1]{Bardos}.
   Assumption $2.$ is also easy to show.
   We leave out the explicit computation
   since it follows the same lines of 
   the computation below for the proof of \eqref{eq:51}.
     
  To prove \eqref{eq:51}, we take in Lemma \ref{lemma_bardos}
  $l_N = j_{N,n}-j_n$, for a fixed $n$, where
  \begin{align*}
  j_{N,n} &=   \frac{1}{2\ii} \int_{\RR^{(N-n)d}} 
\Big(\overline{\Psi_N(x^n, x_{n+1}^N)}\nabla_{x^n} \Psi_N(x^n, x_{n+1}^N) 
  - (\nabla_{x^n}\overline{\Psi_N(x^n, x_{n+1}^N)}) \Psi_N(x^n, x_{n+1}^N \Big) \dd x_{n+1}^N
  ,\\
  j_n &=  \frac{1}{2\ii}\Big( \overline{\psi^{\otimes n}(x^n)} \nabla_{x^n}\psi^{\otimes n}(x^n) 
  - (\nabla_{x^n}\overline{\psi^{\otimes n}(x^n)})\nabla_{x^n}\psi^{\otimes n}(x^n) \Big)
  ,
  \end{align*}
  and
   $$F_N(x^n,y^n) = \nabla_{x^n}\rho_{N,n}^Q(x^n,y^n) - \nabla_{x^n}\rho_{n}^Q(x^n,y^n).$$

   To carry out the verification of the hypothesis $1.$ and $2.$
   of Lemma~\ref{lemma_bardos},
   it is convenient to introduce some notations.
   Let us define
   \begin{align*}
      T^1_N(x^n,y^n) &= \frac{1}{2\ii}\Big(\int_{\RR^{(N-n)d}} 
      \overline{\Psi_N(x^n, x_{n+1}^N)}\nabla_{y^n} \Psi_N(y^n, x_{n+1}^N)  \dd x_{n+1}^N
      - \overline{\psi^{\otimes n}(x^n)} \nabla_{y^n}\psi^{\otimes n}(y^n)\Big)
     ,\\
     T^2_N(x^n,y^n) &= \frac{1}{2\ii}\Big(\int_{\RR^{(N-n)d}} 
      \nabla_{x^n}\overline{\Psi_N(x^n, x_{n+1}^N)} \Psi_N(y^n, x_{n+1}^N)  \dd x_{n+1}^N
      - \nabla_{x^n}\overline{\psi^{\otimes n}(x^n)} \psi^{\otimes n}(y^n)\Big)
     ,
   \end{align*}
    and, similarly,
    \begin{align*}
      t^1_N(x^n) &= \frac{1}{2\ii}\Big(\int_{\RR^{(N-n)d}} 
      \overline{\Psi_N(x^n, x_{n+1}^N)}\nabla_{x^n} \Psi_N(x^n, x_{n+1}^N)  \dd x_{n+1}^N
      - \overline{\psi^{\otimes n}(x^n)} \nabla_{x^n}\psi^{\otimes n}(x^n)\Big)
     ,\\
     t^2_N(x^n) &= \frac{1}{2\ii}\Big(\int_{\RR^{(N-n)d}} 
      (\nabla_{x^n}\overline{\Psi_N(x^n, x_{n+1}^N)}) \Psi_N(x^n, x_{n+1}^N)  \dd x_{n+1}^N
      - (\nabla_{x^n}\overline{\psi^{\otimes n}(x^n)}) \psi^{\otimes n}(y^n)\Big)
     .  
    \end{align*}
    We want to apply Lemma~\ref{lemma_bardos} with the pair $F_N$, $l_N$ taken to be 
    $T_N^1$, $t_N^1$ and $T_N^2$, $t_N^2$.
    Note that $T^1_N$ and $T^2_N$ are trace class operators.
    To see this first note that, for example,
    $\overline{\Psi_N(x^n, x_{n+1}^N)}\nabla_{y^n} \Psi_N(y^n, y_{n+1}^N)$
    is the kernel of a rank-one operator, and rank-one operators are trace class.
    Therefore it is the kernel of a trace class operator.
    Moreover the partial traces (i.e. the integrals in the definitions of $T^1_N$ and $T^2_N$)
    map trace-class operators to trace-class operators (cf.\ e.g.\ \cite[Corollary 2.2]{Bardos}).
    Hence $T^1_N$ and $T^2_N$ are indeed trace class operators on the Hilbert space $L^2(\RR^{nd})$.
    Let for the moment consider $F_N = T^1_N$, $l_N=t^1_N$.
    We have to show that the conditions $1.$ and $2.$ of the Lemma~\ref{lemma_bardos} are satisfied.
    
    We have    
    \begin{align*}
        &\int \Big| \int_{\RR^{(N-n)d}} 
      \Big( \overline{\Psi_N(x^n + h, x_{n+1}^N)}\nabla_{x^n} \Psi_N(y^n, x_{n+1}^N) 
      - 
      \overline{\Psi_N(x^n, x_{n+1}^N)}\nabla_{x^n} \Psi_N(x^n, x_{n+1}^N) \Big)  \dd x_{n+1}^N\Big| \dd x^n =\\
      &\le \int \int_{\RR^{(N-n)d}} 
      \Big| \overline{\Psi_N(x^n + h, x_{n+1}^N)}\nabla_{x^n} \Psi_N(x^n, x_{n+1}^N) 
      - 
      \overline{\Psi_N(x^n, x_{n+1}^N)}\nabla_{x^n}  \Psi_N(x^n, x_{n+1}^N)  \Big|\dd x^N \\
      &\le \int 
      \Big| \overline{\Psi_N(x^n + h, x_{n+1}^N)} 
      - 
      \overline{\Psi_N(x^n, x_{n+1}^N)}\Big|   \, \Big|\nabla_{x^n} \Psi_N(x^n, x_{n+1}^N)  \Big|\dd x^N\\
      &\le  \Big\| \nabla_{x^n}\Psi_N \Big\|_{L^2(\RR^{Nd})}   
      \Big(\int\Big| \overline{\Psi_N(x^n + h, x_{n+1}^N)}  
      - 
      \overline{\Psi_N(x^n, x_{n+1}^N)}\Big|   \dd x^N \Big)^{1/2}\\
      &\le |h|_{\RR^{nd}} \Big\| \nabla_{x^n}\Psi_N \Big\|_{L^2(\RR^{Nd})}^2\\
      &\le |h| \| \Psi_N \|^2_{H^1(\RR^{Nd})}
      ,
    \end{align*}
    where in the third inequality the Cauchy-Schwartz has been used and the fourth inequality
    follows from an $L^2$ mean-value theorem (cf.\ \cite[Corollary 3.2]{Bardos}).
    From the above computation it is easy to see that the pair $T^1_N$ 
    and $t^1_N$ satisfies condition $2.$ of the Lemma~\ref{lemma_bardos}.

    We now show that $T^1_N$ also satisfies conditions $1.$,
    i.e.\ for any compact operator $\mathtt K$ we have
    \begin{equation}\label{eq:1000}
    \Tr(T^1_N \mathtt K) \rightarrow 0,\quad N\rightarrow\infty
    .
    \end{equation}
    Indeed, by taking $\mathtt K=f\otimes\varphi$,
    with $f\in L^2(\RR^{nd})$ and 
    $\varphi\in\mathcal D(\RR^{nd})$ and integrating by parts
    we have
    \begin{align*}
        \Tr(T^1_N \mathtt K)
        &= \int \nabla_{x^n}\rho^Q_{N,n}(x^n,y^n)
        f(y^n)\varphi(x^n)\dd x^n\dd y^n\\
        &= -\int \rho^Q_{N,n}(x^n,y^n) 
        f(y^n)\nabla_{x^n}\varphi(x^n)\dd x^n\dd y^n\\
        &\rightarrow 0
        ,
    \end{align*}
    where the last line follows from the weak$\ast$ convergence
    of $\rho^Q_{N,n}$ proved in \cite[Theorem 4.1]{Bardos}.
    Then note that we can extend the previous result to a 
    $\varphi\in L^2(\RR^{nd})$ 
    by the following Cauchy-Schwartz bound,
    which is uniform in $N$:
    \begin{align*}
        \Tr(T_N^1 f\otimes\varphi)
        &\le \| T_N f\|_{L^2(\RR^{nd})}
        \, \| \varphi\|_{L^2(\RR^{nd})}\\
        &\le  \| T_N\|_{\mathcal B(L^2(\RR^{nd}))} \, 
        \| f\|_{L^2(\RR^{nd})}
        \,\| \varphi\|_{L^2(\RR^{nd})}\\
        &\le  \Tr( |T_N| ) \, 
        \| f\|_{L^2(\RR^{nd})}
        \,\| \varphi\|_{L^2(\RR^{nd})}\\
        &\le  (\| \nabla\Psi_N\|_{L^2(\RR^{Nd})}\,
        \| \Psi_N\|_{L^2(\RR^{Nd})}
        + \| \nabla\psi^{\otimes n}\|_{L^2(\RR^{nd})}\,
        \| \psi^{\otimes n}\|_{L^2(\RR^{nd})})
          \, 
        \| f\|_{L^2(\RR^{nd})}
        \,\| \varphi\|_{L^2(\RR^{nd})}\\
        &\le C \| f\|_{L^2(\RR^{nd})}
        \,\| \varphi\|_{L^2(\RR^{nd})}
        ,
    \end{align*}
    where the first inequality is Cauchy-Schwarz,
    the second is standard,
    the third follows from the fact that the 
    trace-norm is stronger than the operator norm, and
    the fourth follows by the definition of $T^1_N$
    and Cauchy–Schwarz inequality.
    Finally, the fact that $C$ in the last line 
    does not depend on $N$ follows from
     the normalization of the wave functions and
    \cite[Proposition 3.2]{Bardos}.
    
    We also note that the operators $f_1\otimes f_2$
    $f_1,f_2\in L^2(\RR^{nd})$
    are Hilbert-Schmidt, rank-one operators which are dense 
    in the space of compact operators with respect to
    the operator norm.
    Then \eqref{eq:1000}
    indeed holds for any compact operator $K$
    since we can approximate any compact operator
    by rank-$1$ operators in the norm topology
    and the following uniform bound holds
    \begin{align*}
        \Tr(T^1_N \mathtt K)
        &\le \Tr(T^1_N) \|\mathtt K\|_{\mathcal B(L^2(\RR^{nd}))}
        \le C \|\mathtt K\|_{\mathcal B(L^2(\RR^{nd}))}
        ,
    \end{align*}
    where the first inequality is standard,
    and the last line follows as before from 
    the normalization of the wave functions and
    \cite[Proposition 3.2]{Bardos}.
This concludes the prove of \eqref{eq:1000}.
We now can repeat the same argument for $T^2_N$
    since it is the Hilbert-adjoint of $T^1_N$.
Since we have proved that $F_{N,n} = T_N^1+ T_N^2$
    satisfies the hypothesis of Lemma~\ref{lemma_bardos}, 
    by applying it the statement follows.
\end{proof}
\begin{remark}\label{rem:weakprob}
Let us note that $\rho_{N,n},\rho^{\otimes n}$
are densities, with respect of the Lebesgue measure $\dd x^n$ 
on $\RR^{nd}$,
of probability measures $\dd\mu_{N,n}=\rho_{N,n}\dd x^n$,
$\dd\mu_n=\rho^{\otimes n}\dd x^n$, respectively.
Therefore we can rephrase \eqref{eq:50}
saying that $\dd\mu_{N,n}$ converges to $\dd\mu_n$ 
weakly in the sense of probability measures. 
\end{remark}

\subsection{Convergence of the Kinetic energy}
We introduce the total energy associated with the solutions to the linear and nonlinear Schr\"odinger equations, respectively, because we plan to prove the convergence of the one-particle kinetic energy to the corresponding asymptotic one.
\begin{definition}\label{def_energy}
  Let $\Psi_N$ be a solution of the $N$-particle linear Schr\"odinger equation \eqref{schro_intro} and
  $\psi$ be a solution of the one-particle nonlinear Schr\"odinger equation \eqref{nonlin_schro_intro}.
  We denote by $\mathcal E_N$, respectively $\mathcal E$, the quantum energy associated
  with $\Psi_N$, respectively, $\psi$. Explicitly we let
  \begin{align*}
    \mathcal E_N &\deq \mathcal K_N + \mathcal V_N ,
                   \quad
    \mathcal E_\infty \deq \mathcal K_\infty + \mathcal V_\infty
                        ,
  \end{align*}
  where 
  \begin{align*}
    \mathcal K_N(t)
    &\deq \int_{\RR^{Nd}} \tfrac12|\nabla\Psi_N(t,x)|^2 \dd x^N ,
      \quad
      \mathcal V_N
      \deq \frac{1}{N}\sum_{1\le j<k\le N} \int_{\RR^{Nd}} V(|x_j-x_k|) \rho_N(x^N)\dd x^N,
    \\
    \mathcal K(t)   
    &\deq \int_{\RR^{nd}} \tfrac12|\nabla\psi(t,x)|^2 \dd x,
      \hspace{2.4em}
      \mathcal V
      \deq \int_{\RR^d\times \RR^d} |\psi(x)|^2V(|x-y|)|\psi(y)|^2\dd x\dd y
      ,
  \end{align*}
\end{definition}
\begin{claim}\label{total_energy_conservation}
The total quantum energy $\mathcal E_N\deq \mathcal K_N + \mathcal V_N$ of the $N$-particle system
as well as the total quantum energy $\mathcal E \deq \mathcal K + \mathcal V$
of the nonlinear Schr\"odinger equation
are conserved along the time.
\end{claim}
The following useful result, stating that the finite energy condition in the definition of conservative diffusion (Definition \ref{conservativediffusiondefinition}) is a kinetic finite energy condition (see  \cite{carlen1984conservative}), holds.
\begin{proposition}\label{pro:Carlen_kinetic}
    Let $\PP_N$, $\PP$ be the probability,laws of the Nelson diffusions $X_N$, $X$, respectively.
    defined in \eqref{eq:44}. Let $\Psi_N$ be the solution of the $N$ -particle Schr\"odinger equation associated with the Nelson diffusion $X_N$ and $\psi$ be the solution of the nonlinear one-particle Schr\"odinger equation associated  with the Nelson diffusion $X$.
    Moreover, let $u_N,v_n,\rho_N$ and $u,v,\rho$ be as in 
    \eqref{eq:42}, \eqref{eq:43} and \eqref{eq:44}.
    Then, with the notation of the previous definition, we have
    \begin{align*}
    \mathcal K_N(t)
    &= \int \tfrac12( |u_N|^2 + |v_N|^2)\rho_N \dd x^N ,
    \\
    \mathcal K(t)   
    &= \int \tfrac12( |u|^2 + |v|^2)\rho \dd x
      ,
  \end{align*}
\end{proposition}
\begin{proof}
    The simple computations are left to the reader.
\end{proof}
Regarding the kinetic energy part, we need the following convergence property
which says that the one-particle kinetic energy of a system of $N$-particles evolving
according to the (linear) Schr\"odinger equation converges for $N\rightarrow\infty$ 
to the kinetic energy of the one-particle nonlinear Schr\"odinger equation.
For convenience, in the following proposition we take assumptions on the interaction potential that are
similar to those in [\cite{Bardos}, Theorem 5.4].
\begin{proposition}\label{pro:kinetic}
  Let us assume that $V\in L^\infty(\RR_+)$, $V(r)\rightarrow 0$, $r\rightarrow\infty$.
  Then
  \begin{align*}
    \frac{1}{N}\mathcal K_N(t) \rightarrow \mathcal K(t)
    ,
  \end{align*}
  for almost every $t\in\RR_+$. 
\end{proposition}
\begin{proof}
  We first prove the convergence of the one-particle potential-energy
  $\frac{1}{N}\mathcal V_N(t)$, $t$-almost surely, to the limit one-particle potential-energy $\mathcal V(t)$.
  The idea is to obtain, by the conservation of total energy (Claim \ref{total_energy_conservation}),
 the convergence of the one-particle kinetic-energy
  $\frac{1}{N}\mathcal K_N(t)$, $t$-almost surely, to the one-particle limit kinetic-energy $\mathcal K(t)$.
 We remove for the moment the dependence on $t$ in the notation and consider all time-dependent quantities at a fixed time $t$.

  We want to show that
  \begin{equation}\label{eq:52}
    \frac{1}{N} \mathcal V_N \rightarrow \mathcal V
    .
  \end{equation}
  Note that, by definition of $\rho^{\otimes n}$ given in Proposition \ref{prop:factorization}, we have
  \begin{align*}
    \mathcal V
    \deq \int_{\RR^{2d}} \rho^{\otimes 2}(x_1,x_2)V(|x_1-x_2|)\dd x_1\dd x_2
    .
  \end{align*}
  By symmetry, i.e. by indistinguishability  of the $N$ Bose particles, we also have 
  \begin{align*}
    \mathcal V_N
    &= \frac{(N-1)(N-2)}{N}\int
      V(|x_1-x_2|) \rho_{N,2}(x_1,x_2;t)
      \dd x_1 \dd x_2
      ,
  \end{align*}
   where $\rho_{N,2}$ is as in Definition \ref{marginal_Madelung_variables}.
   We have, as in \cite[(4.18)]{Bardos}, that
  \begin{align*}
    \int_{\RR^{2d}} \chi(x,y)\int_{\RR^d} V(|x-z|)\rho_{N,2}^Q(x,z;y,z)\dd z \dd x \dd y
    \rightarrow
    \int_{\RR^{2d}} \chi(x,y)\int_{\RR^d} V(|x-z|)\rho_{2}^Q(x,z;y,z)\dd z \dd x \dd y
    ,
  \end{align*}
  where $\chi$ is a test function in $\mathcal D(\RR^{2d})$.
  It is easy to see that we can now apply Lemma~\ref{lemma_bardos} with
  $F_N(x,y) = \int_{\RR^d} V(|x-z|)\rho_{N,2}^Q(x,z;y,z)\dd z $ and $l_N(x) = F_N(x,x)$.
  The condition $2.$ follows easily by our hypothesis $V\in L^\infty$.
  Hence, by Lemma~\ref{lemma_bardos}, we obtain that
  \begin{align*}
    \int_{\RR^d} \varphi(x)\int_{\RR^d} V(|x-z|)\rho_{N,2}(x,z)\dd z \dd x
    \rightarrow \int_{\RR^{2d}} \varphi(x)\int_{\RR^d} V(|x-z|)\rho_{N,2}(x,z)\dd z \dd x
    ,
  \end{align*}
  for any $\varphi\in C_b(\RR^{d})$.
  Now we note that by Remark~\ref{rem:weakprob} the family of probability
  measures $\rho_{N,n}(x^n,t)\dd x^n$ weakly converges in the
   sense of probability measures for any fixed time $t$.
  As a consequence this family of probability measures
  is tight. Therefore, by tightness and by the boundedness of the potential $V$, we can let 
  $\varphi\uparrow 1$, that is, we can replace $\varphi$ 
  by the constant function equal to $1$.
  Finally, noting that $\frac{1}{N}\frac{(N-1)(N-2)}{N}\rightarrow 1$ for $N\rightarrow 0$,
  we have obtained that indeed $\frac{1}{N}\mathcal V_N\rightarrow \mathcal V$.

By Claim \ref{total_energy_conservation}, assuming that $\frac{1}{N}\mathcal E_N\rightarrow \mathcal E$ at $t=0$,
it remains true for all subsequent times.
Since we proved that, for every fixed time $t$, $\frac{1}{N}\mathcal V_N(t)\rightarrow V_\infty(t)$,
we also have that, for any fixed $t\ge0$,
$\frac{1}{N}\mathcal K_N\rightarrow \mathcal K_\infty$.
The proof is now complete.
\end{proof}
\begin{remark}
    In the stationary case, the separate convergence of the different components (kinetic and potential) of the total ground-state energy is typical for the mean-field or intermediate scaling limit (see, e.g., \cite{albeverio2022mean}). On the contrary, 
    in the
    Gross-Pitaevskii scaling limit, 
    the one-particle kinetic energy does not converge any more to the kinetic energy of the one-particle nonlinear Gross-Pitaevskii equation (see \cite{Li-Se} and references therein).
\end{remark}
A monotonicity property of the relative entropy between marginal laws are finally provided.
\begin{proposition}\label{pro:rele_bound}
  Let $\PP_N$ denotes the law of the $N$ interacting diffusions and 
  let $[\PP_N]_n$ be the push-forward of $\PP_N$ by the projection $\pi_n:\RR^N\rightarrow\RR^n$. Let us assume that $\rho_N(0)\rightarrow\rho_\infty(0)$ as $N \rightarrow \infty$. 
  Then, 
  \begin{align*}
    \KL(\PP_{N,n}|\WW^{\otimes n})\le \KL([\PP_N]_n|\WW^{\otimes n})
  \end{align*}
  and
  \begin{align*}
      \limsup_{N\rightarrow\infty}
      \KL(\PP_{N,n}|\WW^{\otimes n})
\le \KL(\PP_\infty|\WW^{\otimes n})
.
  \end{align*}
\end{proposition}
\begin{proof}
  Note that we have 
  \begin{align*}
    \mathcal H(\PP_{N,n}|\WW^{\otimes n})
    &=  \int_0^T \int |b_{N,n}|^2 \rho_{N,n}(t) \dd x_1^n \dd t\\
    &=  \int_0^T \int |b_{N,n}|^2 \rho_N(t) \dd x^N \dd t\\
    &= \int_0^T \EE\Big[ \big| \EE\Big[ b_N(X^N) \Big| X_1^n(t)\Big]\big|^2 \Big]\\
    &\le \int_0^T \EE\Big[ \EE\Big[ \big|b_N(X^N)\big|^2 \Big| X_1^n(t)\Big] \Big]\\
    &=\int_0^T \EE\Big[ \big|b_N(X^N)\big|^2\Big]\\
    &=\mathcal H(\PP_N|\WW^{\otimes N})
      ,
  \end{align*}
  where the first and last equality follows from Proposition~\ref{proposition_follmer}
  and
  in the fourth line the Jensen inequality has been applied.
  This proves the first statement.

  To prove the second statement we apply Lemmaa \ref{lem:vu} and Proposition \ref{pro:kinetic} obtaining
  \begin{align*}
    \KL(\PP_N|\WW^{\otimes N})
    = \int_0^T \mathcal K_N(t)\dd t  +  H(\rho_N(0)) - H(\rho_N(T))
    .
  \end{align*}
  By hypothesis we have that $\frac{1}{N}\rho_N(0)\rightarrow\rho_\infty(0)$ and
 by Proposition \ref{pro:entropy-limsup} tht
  $\limsup_N \frac{1}{N}H(\rho_N(T))\le H(\rho_\infty(T))$.
  Since, by Proposition \ref{pro:kinetic} it holds that
  $\lim_N \frac{1}{N}\mathcal K_N(t) = \mathcal K_\infty(t)$,
 we obtain
  \begin{align*}
    \limsup_N\frac{1}{N} \KL(\PP_N|\WW^{\otimes N})
    &\le \int_0^T \mathcal K_\infty(t)\dd t  +  H(\rho_\infty(0)) - H(\rho_\infty(T))\\
    &= \KL(\PP_\infty|\WW)
    , 
  \end{align*}
  where the last line follows, as before, by Proposition~\ref{proposition_follmer}.
  This concludes the proof.
\end{proof}

\subsection{The time dependent BEC: convergence of the Nelson process law}
Starting by the (linear) Sch\"odinger equation for an $N$-particle system given in \eqref{schro_intro},
we have seen in Section 4 that there corresponds a $N$-interacting diffusion system given in
\eqref{eq:44}.
We denote by $\mathbb{P}_N $ the law of the process $X_N$ on the path-space $\Omega =C([0,T],\R^{Nd})$.
The time marginals of $X_N,$ i.e. the law of $X_N(t)$, $t\in[0,T]$,
is absolutely continuous with respect to the Lebesgue measure with probability density
$\Law(X_N(t))=\rho_N(x^N,t)\dd x^N=|\Psi_N(x^N,t)|^2\dd x^N$.
Analogously, in \eqref{eq:44} we have introduced the Nelson diffusion $X$ corresponding to the one-particle
nonlinear Schrodinger equation \eqref{nonlin_schro_intro}.
Furthermore, we denote by $\PP_{N,n}$ the probability law of the conditional Nelson diffusion $X_{N,n}$ defined  in \eqref{eq:47}.
More precisely,
by Theorem \ref{Carlentheorem} and Remark \ref{rem:rho_u_Nn}, this law is given by $\Law(X_{N,n}(t)=\rho_{N,n}(t,x^n)\dd x^n$.
Let us finally consider the process $X_{\infty,n}=(X_1,X_2, \dots, X_n),$ where each $X_i$ is independent from
any other $X_j$, $j\ne i$, and has the same law as $X$. Denoting by $\PP_{\infty,n}$
the law of $X_{\infty,n}$ on the path space $\Omega$,
we have that $\Law(X_{\infty,n}(t))=\rho^{\otimes n}(x,t)\dd x^n$. 

The following is our main convergence theorem, which goes beyond the quantum mechanical framework.
\begin{theorem}\label{convergence_path_space}
  The conditioned Nelson diffusion $X^{N,n}$  converges to the Nelson diffusion $X_{\infty,n}$, where the convergence is understood as the weak convergence of the probability measure $\PP_{N,n}$  to $\PP_{\infty,n}$  on the path space $\Omega$. 
\end{theorem}
\begin{proof}
Since we want to employ Theorem \ref{Th:DvR}, we proceed to verify the hypothesis of that theorem
in the special case of the Nelson diffusions considered in the statement of the present theorem.
\begin{enumerate}
\item Let $X_{N,n}$ be the conditional Nelson diffusion defined as in \eqref{eq:47} 
of Section \ref{subsec:Nelsondiffusions}.
By Proposition \ref{Nelsondiffusions} we know that $X_{N,n}$ is well defined.
Moreover by Proposition~\ref{proposition_follmer} we have
\begin{align*}
    \mathcal H(\PP_N|\WW^{\otimes N}) &= \EE[\int_0^t |b_{N}(X^N,t)|^2 \dd t ] ,\\
    \mathcal H(\PP_{N,n}|\WW^{n}) &= \EE[\int_0^t |b_{N,n}(X_{N,n},t)|^2 \dd t ] ,\\
    \mathcal H(\PP|\WW) &= \EE[\int_0^t |b(X,t)|^2 \dd t ] .
\end{align*}
In the special case of Nelson diffusion in the sense of Section \ref{subsec:Nelsondiffusions},
thanks to Lemma \ref{lem:vu} 
we have furthermore that
\begin{align*}
    \mathcal H(\PP_N|\WW^{\otimes N}) &= \int_0^T \mathcal K_N(t)\dd t + H(\rho_N(0))-H(\rho_N(T)),\\
    \mathcal H(\PP|\WW) &= \int_0^T \mathcal K(t)\dd t + H(\rho(0))-H(\rho(T)),\\
\end{align*}
where $\mathcal K_N$ and $\mathcal K$ are given in Definition~\ref{def_energy}
and $\rho_N$, $\rho$ are given in \eqref{eq:42}.
In particular $\mathcal H(\PP_N|\WW^{\otimes N})<+\infty$, 
$\mathcal H(\PP|\WW)<+\infty$.
Moreover by Proposition \ref{pro:rele_bound}, we have
\begin{align*}
    \limsup_N \mathcal H(\PP_{N,n}|\WW^{\otimes n}) &\le \mathcal H(\PP^{\otimes n}|\WW^{\otimes n})
    .
\end{align*}
Hence condition $(i)$ of Theorem \ref{Th:DvR} is satisfied.

\item By Proposition \ref{cor:Bardos_convergence} we have that
the marginal probability densities $\rho_N(t)$ converge in distribution to $\rho_\infty(t),$ for all $t\in [0,T]$. Hence hypothesis $(ii)$ of Theorem \ref{Th:DvR} is satisfied.
Moreover, 
note that, for any $K: \R^{nd}\times [0,T] \rightarrow \R$,  continuous bounded function,
\begin{align*}
    \lim_{N\rightarrow\infty}\EE[\int_0^t \langle b_{N,n}(X_{N,n}(s),s), K(X_{N,n},s)\rangle \dd s]
    &= \lim_{N\rightarrow\infty}\int_0^t \int_{\RR^{nd}} 
    \langle b_{N,n}(x_1^n,s), K(x_1^n,s)\rangle  
    \rho_{N,n}(x_1^n,s)\dd x_1^n \dd s\\
    &=\lim_{N\rightarrow\infty}\int_0^t \int_{\RR^{nd}} 
    \langle \nabla_1^n\rho_{N,n}(x_1^n,s) + j_{N,n}(x_1^n,s) , K(x_1^n,s) \rangle
    \dd x_1^n \dd s\\
    &=\int_0^t \int_{\RR^{nd}} 
    \langle \nabla_1^n\rho^{\otimes n}(x_1^n,s) + j_{\infty,n}(x_1^n,s) , K(x_1^n,s) \rangle
    \dd x_1^n \dd s\\
    &=   \EE[\int_0^t \langle b_{\infty}(X_{\infty}(s),s), K(X_{\infty},s)\rangle\dd s]
    ,
\end{align*}
where in the third line we used Proposition \ref{cor:Bardos_convergence} and in the last line we have
denote by $b_{\infty,n}$ the drift of the process $X_{\infty,n}$, i.e. $b_{\infty,n}=(b,\dots,b)$ is a vector
of $n$ copies of the drift $b$.
In particular, for $n=1$, it holds
\begin{align*}
    \lim_{N\rightarrow\infty}\EE[\int_0^t \langle b_{N,1}(X_{N,1}(s),s), K(X_{N,1},t)]
    =   \EE[\int_0^t \langle b_{\infty}(X_{\infty}(s),s), K(X_{\infty},s)]
    ,
\end{align*}
which implies that the laws $\PP_{N,1}$, $\PP_\infty$ satisfy also the hypothesis $(iii)$ of Theorem~\ref{Th:DvR}.
Therefore, applying the theorem, we obtain the desired result.
  \qedhere
\end{enumerate}
\end{proof}

\bibliography{references} 

\begin{thebibliography}{10}

\bibitem{albeverio2020strong}
S.~Albeverio, F.~C. De~Vecchi, A.~Romano, and S.~Ugolini.
\newblock Strong {K}ac’s chaos in the mean-field {B}ose--{E}instein
  condensation.
\newblock {\em Stochastics and Dynamics}, 20(05):2050031, 2020.

\bibitem{albeverio2022mean}
S.~Albeverio, F.~C. De~Vecchi, A.~Romano, and S.~Ugolini.
\newblock Mean-field limit for a class of stochastic ergodic control problems.
\newblock {\em SIAM Journal on Control and Optimization}, 60(1):479--504, 2022.

\bibitem{albeverio2017entropy}
S.~Albeverio, F.~C. De~Vecchi, and S.~Ugolini.
\newblock Entropy chaos and {B}ose--{E}instein condensation.
\newblock {\em Journal of Statistical Physics}, 168:483--507, 2017.

\bibitem{albeverio2015doob}
S.~Albeverio and S.~Ugolini.
\newblock A {D}oob h-transform of the {G}ross--{P}itaevskii {H}amiltonian.
\newblock {\em Journal of Statistical Physics}, 161:486--508, 2015.

\bibitem{bao2023ergodic}
X.~Bao and S.~Tang.
\newblock Ergodic control of {M}ckean--{V}lasov {SDE}s and associated {B}ellman
  equation.
\newblock {\em Journal of Mathematical Analysis and Applications},
  527(1):127404, 2023.

\bibitem{ErdosBardos}
C.~Bardos, L.~Erdös, F.~Golse, N.~Mauser, and H.-T. Yau.
\newblock Derivation of the {S}chr\"odinger--{P}oisson equation from the
  quantum n-body problem.
\newblock {\em I}, 334:515--520, 03 2002.

\bibitem{boccato2018complete}
C.~Boccato, C.~Brennecke, S.~Cenatiempo, and B.~Schlein.
\newblock Complete {B}ose--{E}instein condensation in the {G}ross--{P}itaevskii
  regime.
\newblock {\em Communications in Mathematical Physics}, 359:975--1026, 2018.

\bibitem{Braun}
W.~Braun and K.~Hepp.
\newblock The {V}lasov dynamics and its fluctuation in the 1/n limit of
  interacting particles.
\newblock {\em {\em Comm.Math. Phys.} {\bf 56}}, 6:101--113, 1997.

\bibitem{Bardos}
{C. Bardos}, {F. Golse}, and {N. Mauser}.
\newblock Weak coupling limit of the {N}-particle {S}chrödinger equation.
\newblock {\em Methods and Applications of Analysis}, 7(2):275--294, June 2000.

\bibitem{carlen2014stochastic}
E.~Carlen.
\newblock Stochastic mechanics: a look back and a look ahead.
\newblock {\em Diffusion, quantum theory and radically elementary mathematics},
  47:117--139, 2014.

\bibitem{CarlenCD}
E.~A. Carlen.
\newblock Conservative {D}iffusions.
\newblock {\em Communications in Mathematical Physics}, 94(3):293--315, 1984.

\bibitem{carlen1984conservative}
E.~A. Carlen.
\newblock {\em Conservative {D}iffusions: A constructive approach to {N}elson's
  stochastic mechanics}.
\newblock Princeton University, 1984.

\bibitem{Carlen:162610}
E.~A. Carlen.
\newblock {Existence and sample path properties of the diffusions in {N}elson's
  stochastic mechanics}.
\newblock Technical report, Bielefeld TU. Bielefeld-Bochum-Stochastik,
  Bielefeld, 1984.

\bibitem{carmona2018probabilistic}
R.~Carmona, F.~Delarue, et~al.
\newblock {\em Probabilistic theory of mean field games with applications
  I-II}.
\newblock Springer, 2018.

\bibitem{Cattiaux2021TimeRO}
P.~Cattiaux, G.~Conforti, I.~Gentil, and C.~L'eonard.
\newblock Time reversal of diffusion processes under a finite entropy
  condition.
\newblock {\em Annales de l'Institut Henri Poincar{\'e}, Probabilit{\'e}s et
  Statistiques}, 2021.

\bibitem{petronimorato}
N.~Cufaro~Petroni and L.~M. Morato.
\newblock Entangled states in stochastic mechanics.
\newblock {\em Journal of Physics A: Math. Gen.}, 33:5833--5848, 2000.

\bibitem{DeVecchiUgolini}
F.~De~Vecchi and S.~Ugolini.
\newblock An entropy approach to {B}ose-{E}instein condensation.
\newblock {\em Communications on Stochastic Analysis}, 8, 12 2014.

\bibitem{DeVecchi_Rigoni2024}
F.~C. De~Vecchi and C.~Rigoni.
\newblock A description based on optimal transport for a class of stochastic
  {M}ckean-{V}lasov control problems.
\newblock {\em arXiv preprint arXiv:2405.12960}, 2024.

\bibitem{Li-Se}
{E.H. Lieb} and {R. Seiringer}.
\newblock Proof of {B}ose-{E}instein condensation for dilute trapped gases.
\newblock {\em Phys. Rev. Lett.}, 88:1--4, 2002.

\bibitem{ESY}
L.~Erdos, B.~Schlein, and H.-T. Yau.
\newblock Derivation of the {G}ross-{P}itaevskii equation for the dynamics of
  {B}ose-{E}instein condensate.
\newblock {\em Ann. of Math. (2)}, 172, 07 2006.

\bibitem{ErdossSchlein}
L.~Erd{\H{o}}s, B.~Schlein, and H.-T. Yau.
\newblock Derivation of the cubic non-linear {S}ch\"odinger equation from
  quantum dynamics of many body system.
\newblock {\em Inventiones mathematicae}, 167:515--614, 2007.

\bibitem{erdos2001derivation}
L.~Erd{\"o}s and H.~Yau.
\newblock Derivation of the nonlinear {S}chr{\"o}dinger equation from a many
  body coulomb system.
\newblock {\em Advances in Theoretical and Mathematical Physics}, 5(6), 2001.

\bibitem{Carlen2010}
{Eric A. Carlen}, {Maria C. Carvalho}, {Jonathan Le Roux}, {Michael Loss}, and
  {Cédric Villani}.
\newblock Entropy and chaos in the {K}ac model.
\newblock {\em Kinetic and Related Models}, 3(1):85--122, 2010.

\bibitem{GuMo}
{F. Guerra} and {L. Morato}.
\newblock Quantization of dynamical systems and stochastic control theory.
\newblock {\em Phys. Rev. D}, 27:1774--1786, 1983.

\bibitem{Follmer_Wiener_space}
H.~F{\"o}llmer.
\newblock Time reversal on {W}iener space.
\newblock In S.~A. Albeverio, P.~Blanchard, and L.~Streit, editors, {\em
  Stochastic Processes --- Mathematics and Physics}, pages 119--129, Berlin,
  Heidelberg, 1986. Springer Berlin Heidelberg.

\bibitem{follmer1988random}
H.~F\"ollmer.
\newblock Random fields and diffusion processes.
\newblock {\em Ecole d'Ete de Probabilites de Saint-Flour XV-XVII, 1985-87},
  1988.

\bibitem{FGS07}
J.~Fr{\"o}hlich, S.~Graffi, and S.~Schwarz.
\newblock Mean-field-and classical limit of many-body schr{\"o}dinger dynamics
  for bosons.
\newblock {\em Communications in mathematical physics}, 271(3):681--697, 2007.

\bibitem{FKP07}
J.~Fr{\"o}hlich, A.~Knowles, and A.~Pizzo.
\newblock Atomism and quantization.
\newblock {\em Journal of Physics A: Mathematical and Theoretical},
  40(12):3033, 2007.

\bibitem{grodnik1970representations}
J.~Grodnik and D.~Sharp.
\newblock Representations of local nonrelativistic current algebras.
\newblock {\em Physical Review D}, 1(6):1531, 1970.

\bibitem{haussmann1986time}
U.~G. Haussmann and E.~Pardoux.
\newblock Time reversal of diffusions.
\newblock {\em The Annals of Probability}, pages 1188--1205, 1986.

\bibitem{GinVel}
{J. Ginibre} and {G. Velo}.
\newblock On a class of nonlinear {S}chr\"{o}dinger equation with nonlocal
  interaction.
\newblock {\em Math. Zeitschr.}, 169:109--145, 1980.

\bibitem{KP10}
A.~Knowles and P.~Pickl.
\newblock Mean-field dynamics: singular potentials and rate of convergence.
\newblock {\em Communications in Mathematical Physics}, 298:101--138, 2010.

\bibitem{PhysRevLett.88.170409}
E.~H. Lieb and R.~Seiringer.
\newblock Proof of {B}ose--{E}instein condensation for dilute trapped gases.
\newblock {\em Phys. Rev. Lett.}, 88:170409, Apr 2002.

\bibitem{MorUgo1}
{L.M. Morato} and {S. Ugolini}.
\newblock Stochastic description of a {B}ose-{E}instein condensate.
\newblock {\em Annales Henry Poincaré}, 12(8):1601--1612, 2011.

\bibitem{loffredo1996eulerian}
M.~I. Loffredo and S.~Ugolini.
\newblock Eulerian versus {L}agrangian variational principles in stochastic
  mechanics.
\newblock {\em Meccanica}, 31:195--206, 1996.

\bibitem{HauMis}
{M. Hauray} and {S. Mischler}.
\newblock On {K}ac's chaos and related problems.
\newblock {\em Journal of Functional Analysis}, 16(7):1423--1466, 2014.

\bibitem{millet1989integration}
A.~Millet, D.~Nualart, and M.~Sanz.
\newblock Integration by parts and time reversal for diffusion processes.
\newblock {\em The Annals of Probability}, pages 208--238, 1989.

\bibitem{minellimorato}
T.~Minelli and L.~M. Morato.
\newblock Position momentum joint probability densities and generating
  function.
\newblock {\em Physics Letter}, 134:285--287, 1989.

\bibitem{morato2013localization}
L.~M. Morato and S.~Ugolini.
\newblock Localization of relative entropy in {B}ose--{E}instein condensation
  of trapped interacting bosons.
\newblock In {\em Seminar on Stochastic Analysis, Random Fields and
  Applications VII: Centro Stefano Franscini, Ascona, May 2011}, pages
  197--210. Springer, 2013.

\bibitem{nam2016ground}
P.~T. Nam, N.~Rougerie, and R.~Seiringer.
\newblock Ground states of large bosonic systems: the {G}ross--{P}itaevskii
  limit revisited.
\newblock {\em Analysis \& PDE}, 9(2):459--485, 2016.

\bibitem{NelsonD}
E.~Nelson.
\newblock Derivation of the {S}chr\"{o}dinger equation from newtonian
  mechanics.
\newblock {\em Phys. Rev.}, 150:1079--1085, 1966.

\bibitem{NelsonQF}
E.~Nelson.
\newblock {\em Quantum Fluctuations}.
\newblock Princeton University Press, Princeton, 1984.

\bibitem{reviewNelson}
E.~Nelson.
\newblock Review of stochastic mechanics.
\newblock {\em Journal of Phys: Conf.Ser.}, 361:012011, 2012.

\bibitem{pavon1989stochastic}
M.~Pavon.
\newblock Stochastic control and nonequilibrium thermodynamical systems.
\newblock {\em Applied Mathematics and Optimization}, 19:187--202, 1989.

\bibitem{pavon1995hamilton}
M.~Pavon.
\newblock Hamilton’s principle in stochastic mechanics.
\newblock {\em Journal of Mathematical Physics}, 36(12):6774--6800, 1995.

\bibitem{pavon2006}
M.~Pavon.
\newblock A stochastic control problem connected to the measurement process in
  stochastic mechanics.
\newblock In {\em Proc. 17th Int. Symp. on Mathematical Theory of Networks and
  Systems (MTNS)}, 2006.

\bibitem{pham2009continuous}
H.~Pham.
\newblock {\em Continuous-time stochastic control and optimization with
  financial applications}, volume~61.
\newblock Springer Science \& Business Media, 2009.

\bibitem{P15}
P.~Pickl.
\newblock Derivation of the time dependent gross--pitaevskii equation with
  external fields.
\newblock {\em Reviews in Mathematical Physics}, 27(01):1550003, 2015.

\bibitem{reddiger2023towards}
M.~Reddiger and B.~Poirier.
\newblock Towards a mathematical theory of the {M}adelung equations:
  {T}akabayasi’s quantization condition, quantum quasi-irrotationality, weak
  formulations, and the {W}allstrom phenomenon.
\newblock {\em Journal of Physics A: Mathematical and Theoretical},
  56(19):193001, 2023.

\bibitem{Schlein2009}
I.~Rodnianski and B.~Schlein.
\newblock Quantum fluctuations and rate of convergence towards mean field
  dynamics.
\newblock {\em Comm. Math. Phys.}, 291(1):31--61, 2009.

\bibitem{S80}
H.~Spohn.
\newblock Kinetic equations from {H}amiltonian dynamics: {M}arkovian limits.
\newblock {\em Reviews of Modern Physics}, 52(3):569, 1980.

\bibitem{QKE}
H.~Spohn.
\newblock Quantum kinetic equations.
\newblock In {\em On the Three Levels}, pages 1--10. Plenum Press, New York,
  1994.

\bibitem{Ugolini}
S.~Ugolini.
\newblock {B}ose-{E}instein condensation: a transition to chaos result.
\newblock {\em Communications on Stochastic Analysis}, 6, 2012.

\bibitem{von2012optimal}
M.-K. von Renesse.
\newblock An optimal transport view of {S}chr{\"o}dinger's equation.
\newblock {\em Canadian mathematical bulletin}, 55(4):858--869, 2012.

\bibitem{Yasue}
K.~Yasue.
\newblock Stochastic calculus of variations.
\newblock {\em J. Funct. Anal}, 41:327--340, 1981.

\end{thebibliography}
\bibliographystyle{abbrv}

\end{document}